\documentclass[11pt]{article}

\usepackage{grffile}

\usepackage{amsmath,amsfonts,amscd,a4wide}
\usepackage{amsmath}
\usepackage{amssymb}
\usepackage{graphicx}
\usepackage{amsthm}
\usepackage{color}

\usepackage{transparent}
\usepackage[skip=10pt,font=footnotesize]{caption}
\usepackage{subcaption}
\usepackage{enumitem}

\newtheorem{Lemma}{Lemma}[section]
\newtheorem{Theorem}{Theorem}
\newtheorem{Proposition}[Lemma]{Proposition}
\newtheorem{Corollary}[Lemma]{Corollary}

\newtheorem{Hypothesis}[Lemma]{Hypothesis}
\newtheorem{Example}{Example}

 {\begin{trivlist} \item[]{\bf Proof. }}%
 {\hspace*{\fill}$\rule{.4\baselineskip}{.4\baselineskip}$\end{trivlist}}

\newenvironment{Acknowledgement}%
 {\begin{trivlist}\item[]\textbf{Acknowledgements.}}{\end{trivlist}}

\makeatletter\@addtoreset{figure}{section}\makeatother

\makeatletter \@addtoreset{equation}{section} \makeatother

\title{Oblique and checkerboard patterns in the quenched Cahn-Hilliard model}
\author{Ryan Goh\\\small Boston University\\\small Department of Mathematics and Statistics\\\small Boston, MA, USA \and Ben Hosek\\\small Boston University\\\small Department of Mathematics and Statistics\\\small Boston, MA, USA\\\small\texttt{bhosek@bu.edu}}
\date{}

\begin{document}

\maketitle

\begin{abstract}
We consider transversely modulated fronts in a directionally quenched Cahn-Hilliard equation, posed on a two-dimensional infinite channel, with both parameter and source-term type heterogeneities. Such quenching heterogeneities travel through the domain, excite instabilities, and can select the pattern formed in their wake. We in particular study striped patterns which are oblique to the quenching direction and checkerboard type patterns. Under generic spectral assumptions, these patterns arise via an $O(2)$-Hopf bifurcation as the quenching speed is varied, with symmetries arising from translations and reflections in the transverse variable. We employ an abstract functional analytic approach to establish such patterns near the bifurcation point.  Exponential weights are used to address neutral continuous spectrum, and a co-domain restriction is used to address neutral mass-flux. We also give a method to determine the direction of bifurcation of fronts. We then give an explicit example for which our hypotheses are satisfied and for which bifurcating fronts can be investigated numerically.
\end{abstract}

\paragraph{Statements and Declarations}\textit{Competing Interests:} The authors have no competing interests to declare.



\begin{Acknowledgement}
    The authors were partially supported by the National Science Foundation through grants NSF-DMS-2006887 (RG, BH) and NSF-DMS-1616064 (BH).
\end{Acknowledgement}

\section{Introduction}
\subsection{Motivation}
The Cahn-Hilliard equation
\begin{equation}
\partial_tu=-\Delta(\Delta u+f(u)), \quad f(u) = u - u^3, \quad u(\mathbf{x},t)\in \mathbb{R},\quad (\mathbf{x},t)\in\mathbb{R}^d\times \mathbb{R},
\end{equation}
is a prototypical and well-studied model for phase separation processes in two-phase systems in a variety of contexts; see for example \cite{Miranville} for a mathematical review with many references. Through different initial conditions and boundary conditions, this equation can exhibit many different types of patterns. In particular, small random initial data, say on the unbounded domain, tends to lead to the formation of a random assortment of layers, stripes, spots, and defects, most of which are unstable via local coarsening. We remark that this equation is a $H^{-1}$ gradient-flow with respect to the following free-energy
\begin{equation}\mathcal{E}[u]=\int_{\Omega}\frac{1}{2}|\nabla u|^2+F(u)d\mathbf{x},\nonumber\end{equation}
defined on a generic domain $\Omega$, with symmetric double-well potential $ F(u) = \frac{1}{4}(1-u^2)^2$ which favors the states $u = \pm1$. 

In several experimental and phenomenological settings we can observe phase separation in a binary phase alloy, or a model thereof \cite{Foard12,Guo,Thiele21,Krekhov,Wilczek_2015}.  In many of these experiments, a process known as directional quenching has been used to induce phase separation in a controlled manner and select the pattern formed in the wake. Indeed, depending on the initial data, and the shape and speed of the quench, a variety of patterns can be formed including  regular spot arrangements, stripes of different orientations and wavenumbers, layers between pure $\pm 1$ states, as well as square and rhomb patterns. Here a quench travels through the spatial domain inciting instability in a given homogeneous equilibrium state, typically by spatiotemporally mediating the potential $F$ between single-well and double-well configuration, the latter of which is given above. In our work, and in parallel with several of the above mentioned references, we consider quenched patterns on a two-dimensional spatial domain. 

Recent work in this area includes \cite{Goh}, which studied quenched stripe formation in one spatial dimension for a variety of quenching and source type heterogeneities with spinodally unstable regions which are bounded in $x$.  The work \cite{Monteiro} studies two-dimensional directional quenched fronts with unbounded quenching domain in both the Allen-Cahn and Cahn-Hilliard equations showing that with zero quenching speed pure phase selection, vertical stripes, horizontal interfaces, and horizontal stripes can be formed, though oblique stripes cannot be formed. Existence results can also be obtained for non-zero quenching speeds for many of the aforementioned structures, but interestingly, these works did not address the formation of stripes which are oblique to the quenching interface in the moving interface case.  Such an ambiguity motivates and is one the main focuses of our work. Results from \cite{GohScheel}, which studied the Swift-Hohenberg equation, showed that weakly oblique stripes exist as perturbations of parallel stripes.  Because of this, and results of numerical simulations (discussed below), we expect oblique stripes to exist in the quenched Cahn-Hilliard equation. Our work also reveals cellular, or ``checkerboard" type patterns which biurcate with oblique stripes.

Following the functional analytic methods of \cite{Goh}, we study the bifurcations of transversely modulated patterns in the presence of quenching terms which have localized or bounded spinodal unstable regimes.  That is we look for bifurcating fronts which are spatially patterned but are still asymptotically constant, with the pattern state lying in a potentially moving localized spatial region. This modeling assumption allows us to focus on the pattern forming dynamics and behavior just behind the quenching line.  One hopes to build upon these results to establish large amplitude patterns, as well as fronts which converge to these patterns asymptotically in the far-field.

We seek to understand how a spatial heterogeneity can select patterns in the Cahn-Hilliard equation in two spatial dimensions under directional quenching. Our assumptions are roughly as follows. We consider nonlinearities of the form $f(x-ct,u)$ with a given front solution $u_*(x-ct)$, both of which converge exponentially fast in the co-moving frame $\tilde x:= x - ct$ with quenching speed $c$, to states $f_\pm(u)$ and $u_\pm$.  Further we assume for $\Tilde{u}_\pm$ close to $u_\pm$ there exist a family of smooth asymptotically constant front solutions asymptotic to $\Tilde{u}_\pm$. 
We further assume a generic transverse Hopf instability of the associated linearization $L$ about $u_*$. In particular, we assume that an isolated, semi-simple pair of complex conjugate eigenvalues with transversely modulated eigenfunctions cross the imaginary axis as the quench speed $c$ varies while there are no other resonant spectrum at integer multiples of the Hopf frequency. As the quenched equation, with heterogeneity varying only in the $x$ direction, possesses a reflection symmetry $y\mapsto -y$ in the vertical direction along the quenching line, we generically assume that the Hopf eigenvalues have algebraic and geometric multiplicity two. As there is also a translation symmetry in $y$, we hence study Hopf-instabilities in the presence of a transverse $O(2)$ symmetry.  We mention the works \cite{barker2021transverse,pogan20152}, which study $O(2)$-Hopf bifurcations in viscous slow magnetohydrodynamic shocks and in viscous shock waves in a channel, respectively, and handle similar problems to ours in different ways.

Under these assumptions, we establish the existence a pair of one-parameter families of time-periodic solutions which bifurcate from the front solution $u_*$. These branches, which take the form of oblique stripe and checkerboard patterns, respectively correspond to ``rotating" and ``standing" waves under the $O(2)$ symmetry group.  Our results also give computable bifurcation coefficients which can be used determine whether these bifurcations are subcritical or supercritical.

The proof is done through an abstract functional analytic approach.  We use exponentially weighted spaces to push neutral continuous spectrum away from the imaginary axis and, along with co-domain restriction which takes into account mass-conservation, we obtain a linearization of \eqref{comoving_inhom} which is a Fredholm operator of index 0.  This allows us to perform a Lyapunov-Schmidt reduction to produce a set of finite-dimensional bifurcation equations, which we can then use to establish bifurcating solutions, obtain their leading order expansions, and determine their bifurcation direction.

The rest of the introduction is devoted to developing our setting, stating our hypotheses, and finally stating the theorem we wish to prove.

\subsection{Our Setting}
We consider the following modified Cahn-Hilliard equation with spatiotemporal heterogeneities in both the nonlinearity $f$ (corresponding to changes in the potential well) as well as a moving source term $\chi$ which adds mass to the system as it travels. We consider heterogeneities which rigidly propagate in the horizontal direction with fixed speed $c$, leaving a front solution its wake. Our equation takes the form, 
\begin{equation}\label{e:ch-1}
\partial_tu=-\Delta(\Delta u+f(x-ct,u))+c\chi(x-ct;c), \quad \mathbf{x} = (x,y)\in\mathbb{R}^2, \,\, t\in \mathbb{R}, \,\,\, \Delta = \partial_x^2+\partial_y^2.
\end{equation}

Before continuing with our general hypotheses on the heterogeneities, front solutions, and their spectra, we give a specific example which will motivate and guide our work.  
\begin{Example} \label{ex:top} {\it Tophat quench}
\end{Example}
Consider a cubic-quintic nonlinearity, 
\begin{equation}\label{e:nlcq}
    f(x-ct,u)=h(x-ct)u+\gamma u^3-u^5,\qquad \gamma \in \mathbb{R},
\end{equation}
with a parameter heterogeneity $h(x-ct)$ on the linear term. Here $h$ is 1 in the interval $[-K+\delta,K-\delta]$, -1 outside the interval $[-K,K]$, and smoothly and monotonically transitions between the two states in between the intervals; see \eqref{e:hex} for an approximate example. 

 When $h\equiv1$, the sign of the parameter $\gamma$, roughly speaking, mediates between supercritical ($\gamma<0$) and subcritical ($\gamma>0$) pattern-forming dynamics. Indeed in this situation, local perturbations of the trivial state $u\equiv0$ will grow, invade, and form patterned states in different manners for these two cases, with the former corresponding to \emph{pulled} front invasion where the linear dynamics about $u\equiv0$ govern front dynamics, and \emph{pushed} front invasion where the nonlinear dynamics behind the interface accelerate invasion.  For both parameter domains, one finds patterned  states in the quenched system for quenching speeds $c$ approximately below the free invasion speed. We remark that in the pushed case such quenches have been observed to form patterns for quenching speeds faster than the free invasion speed; see \cite{goh2016pattern}. 

Figure \ref{fig:ch-ex1} gives snapshots of numerical simulations of such a quenched nonlinearity inducing oblique stripes and checkerboard patterns in a periodic channel in the pulled case, $\gamma = -1$. We note these solutions are time-periodic so that a $y$ cross-section of $u$ resembles a traveling, or rotating, wave in the former and a standing wave in the latter. 
\begin{figure}[htbp]
 \centering
\includegraphics[trim={0.5in 1.5in 0 2.25in},clip]{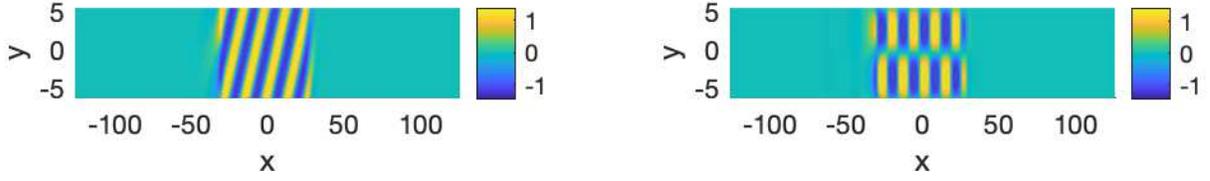}
\vspace{-1in}
    \caption{Numerically, one can find these checkerboards (right) and oblique stripes (left) arising from Example \ref{ex:top}. Here $\gamma = -1$, and patterns arise from the trivial front $u_*=0$.  The quench speed was chosen to be $c=1$. }\label{fig:ch-ex1}
\end{figure}
This example is investigated in more detail in Section \ref{s:ex}.

\paragraph{General Setting}
In the general case, we consider solutions $u(x,y,t)$ of \eqref{e:ch-1} which are periodic in both $t$ and $y$.  In particular we have $t\in[0,2\pi/\omega)$ and $y\in[0,2\pi/k)$, where $\omega$ is the frequency in time and $k$ is the frequency in the second spatial variable.  We then rescale time to the new variable $\tau=\omega t$, we rescale the vertical spatial variable to $\Tilde{y}=ky$, and put the system into the co-moving frame $\Tilde{x}=x-ct$. Simplifying our notation by removing the tildes and writing $\Delta_k:=\partial_x^2+k^2\partial_y^2$, we can then write the Cahn-Hilliard equation as \begin{equation}\label{comoving_inhom}
\begin{split}
     \omega\partial_\tau u&=-\Delta_k(\Delta_k u+f(x,u))+c\partial_xu+c\chi(x;c)\\
    u(x,y,\tau)&=u(x,y,\tau+2\pi)\\
    u(x,y,\tau)&=u(x,y+2\pi,\tau).
\end{split}
\end{equation}
Note that there are symmetries in the $y$ variable. In particular \eqref{comoving_inhom} is invariant under translations  $y\mapsto y+\theta$ and reflections $y\mapsto-y$. Hence any bifurcating fronts will occur in the presence of spatial symmetries.  In particular, the symmetry group will be $O(2)$, being the semidirect product of $SO(2)$ (which contains rotations) with $\mathbb{Z}_2$ (which represents the reflections).  With these preliminaries set, we are now ready to present our hypotheses.

\subsection{Hypotheses and Main Result}
We begin our hypotheses by specifying restrictions on the form of the nonlinearity $f$ and an associated traveling front solution, which propagates with fixed speed, and is asymptotically constant in space, $u_*.$ We remark that these hypotheses are an extension of the one-dimensional setting of \cite{Goh} to the two-dimensional case. The quenching speed $c$ will serve as the main bifurcation parameter of our study.

\begin{Hypothesis}\label{hyp1}
    The nonlinearity $f$ is smooth in both $x$ and $u$, and converges with an exponential rate to smooth functions $f_\pm:=f_\pm(u)$ as $x\to\pm\infty$.  This convergence is uniform for $u$ in bounded sets.
\end{Hypothesis}

\begin{Hypothesis}\label{hyp2}
    There exists a front solution $u_*(x;c_*)$ of \eqref{comoving_inhom} for some $c_*>0$ with \[\lim_{x\to\pm\infty}u_*(x;c_*)=u_\pm.\]
    Moreover, $u_*\in\mathcal{C}^4(\mathbb{R})$ and \begin{equation}|u_*(x)-u_\pm|+\sum_{j=1}^3|\partial_x^ju_*(x)|\leq Ce^{-\beta|x|}\nonumber\end{equation}
    for some $C,\beta>0$.  We refer to this front solution as the primary front.
\end{Hypothesis}

This primary front will be the solution to the Cahn-Hilliard equation from which our patterns bifurcate.  In our previous example, the trivial solution $u_*\equiv0$ plays this role.

Next, we assume that there are no additional neutral modes of the spatial linearization at the origin, and thus that the Hopf instability is the sole neutral mode at the bifurcation speed $c_*$

\begin{Hypothesis}\label{hyp3}
    The point $0\in\mathbb{C}$ is not contained in the extended point spectrum of the linearization $L:H^4(\mathbb{R}\times\mathbb{T})\subset L^2(\mathbb{R}\times\mathbb{T})\to L^2(\mathbb{R}\times\mathbb{T})$ defined as \[Lv=-\Delta_k(\Delta_k v+\partial_uf(x,u_*(x))v)+c_*\partial_xv.\]
\end{Hypothesis}
Furthermore, we assume that fronts persist for perturbations of asymptotic states of the front which preserve the difference between values at $x = \pm\infty.$
\begin{Hypothesis}\label{hyp4}
    Assuming the above hypotheses, for $\Tilde{u}_\pm$ in a small neighborhood of $u_\pm$ with $\Tilde{u}_+-\Tilde{u}_-=u_+-u_-$, and $c$ close to $c_*$, there exists a family of smooth front solutions $u_*(x;c)$ asymptotic to $\Tilde{u}_\pm$ satisfying Hypothesis \ref{hyp2}.
\end{Hypothesis}

This implies that our assumptions are open; that we can vary $u_\pm$ and still be able to find a front solution.  The next hypothesis ensures that there exists a generic Hopf-instability, with transversely modulated eigenfunctions, in the presence of a $y$-reflection symmetry.
\begin{Hypothesis}\label{hyp5}
    The operator $L$ defined on $L^2(\mathbb{R}\times\mathbb{T})$ as above has an isolated pair of eigenvalues,  $\lambda_\pm(c)=\mu(c)\pm i\kappa(c)$, with algebraic and geometric multiplicity two and $L^2(\mathbb{R}\times\mathbb{T})$-eigenfunctions $e^{iy}p(x), e^{-iy}\overline{p(x)}$ along with their images under the reflection symmetry $y\mapsto-y$, such that for some $\omega_*\neq0$ and $c_*>0$ 
    \begin{equation}
    \mu(c_*)=0,\qquad \mu'(c_*)>0\qquad \kappa(c_*)=\omega_*.
    \end{equation}
\end{Hypothesis}

Here we note that the geometrically double eigenvalues are induced by the $y$-reflection symmetry, and the above assumption guarantees that they are semi-simple.  The higher multiplicity of the Hopf modes precludes application of the standard Hopf Theorem and, due to the presence of symmetry, requires the application an equivariant Hopf theorem.   For more information see Appendix \ref{app2}.

Recall our symmetries are the $SO(2)$ action $y\mapsto y+\theta$,  and $\mathbb{Z}_2$ reflection action $y\mapsto-y$.  The action of $SO(2)$ leads to rotating, or traveling, waves, which appear in our system as oblique striped patterns.  The $\mathbb{Z}_2$-action produces standing waves, which appear as checkerboard patterns.

We also define the $L^2$-adjoint eigenfunctions to $p$ and $\overline{p}$, as $\psi_+$ and $\psi_-$ respectively, whose corresponding eigenvalues must have the same algebraic and geometric multiplicity as the eigenvalues $\lambda_\pm(c)$.  We further normalize $\psi_+$ and $\psi_-$ such that $\langle p,\psi_+\rangle_{L^2}=\langle\overline{p},\psi_-\rangle_{L^2}=1$.

Finally, we make a non-resonance assumption which guarantees that there are no point or essential spectrum touching the imaginary axis at frequencies which are non-zero multiples of the Hopf frequency $\omega_*.$ Here $\lambda = 0$ is not included to take into account the neutral continuous spectrum which touches the origin in a quadratic tangency induced by the neutral mass flux/conservation structure of \eqref{e:ch-1}.
\begin{Hypothesis}\label{hyp6}
    For all $\lambda\in i\omega_*(\mathbb{Z}\setminus\{0,\pm1\})$, the operator $L-\lambda$ is invertible when considered on the unweighted space $L^2(\mathbb{R}\times\mathbb{T})$.
\end{Hypothesis}

With these hypotheses, we are now ready to state our result.

\begin{Theorem}\label{thm1}
    Assume Hypotheses \ref{hyp1} through \ref{hyp6}.  Then there exist a pair of one-parameter families of time-periodic solutions which bifurcate from the front solution $u_*(x;c)$ as the speed $c$ varies through $c_*$.
    The bifurcation equation has, to leading order, the form\begin{equation}
        \Theta(a,b;\Tilde{c},\Tilde{\omega})=\left(\lambda_{\Tilde{c}}(0)\Tilde{c}+i\Tilde{\omega}+\frac{\theta_1+\theta_2}{2}N\right)\begin{pmatrix}
            a\\b
        \end{pmatrix}+\frac{\theta_2-\theta_1}{2}\delta\begin{pmatrix}
            a\\-b
        \end{pmatrix},
    \end{equation} 
    with amplitudes $a,b\in\mathbb{C}$, $N=|a|^2+|b|^2, \delta=|b|^2-|a|^2,$ and 
   \begin{align}
        \theta_1&=\left\langle-\Delta_k\left(\frac{1}{2}\partial_u^3f(x,u_*)p^2\overline{p}+\partial_u^2f(x,u_*)[p\phi_{a\overline{a}}+\overline{p}\phi_{aa}]\right),\psi_+\right\rangle_{L^2},\\
        \theta_2&=\left\langle-\Delta_k(\partial_u^3f(x,u_*)p^2\overline{p}+\partial_u^2f(x,u_*)[p\phi_{b\overline{b}}+p\phi_{a\overline{b}}+\overline{p}\phi_{ab}]),\psi_+\right\rangle_{L^2}.
    \end{align} 
    Here, the $\phi_{ij}$'s are particular functions of $x$ which can be obtained using the linear operator and evaluations of derivatives of $f$ on the front $u_*$ (see equations \eqref{phi_a_a} to \eqref{phi_barb_barb}), and $a$ and $b$ are the coordinates of the kernel of \eqref{linearized}, the time-dependent linearization of the Cahn-Hilliard equation. The direction of the bifurcation can be determined by examining the relationship between $\frac{\theta_1+\theta_2}{2}$ and $\frac{\theta_2-\theta_1}{2}$.
    \begin{itemize}
        \item If $\mathrm{Re}\,\left(\frac{\theta_1+\theta_2}{2}\right)>0$, then standing waves (checkerboard patterns) will bifurcate as $c$ increases through $c_*$
        \item If $\mathrm{Re}\,\left(\frac{\theta_1+\theta_2}{2}\right)<0$, then standing waves will bifurcate as $c$ decreases through $c_*$
        \item If $\mathrm{Re}\,\left(\frac{\theta_2-\theta_1}{2}\right)<\mathrm{Re}\,\left(\frac{\theta_1+\theta_2}{2}\right)$, then rotating waves (oblique stripe patterns) will bifurcate as $c$ increases through $c_*$
        \item If $\mathrm{Re}\,\left(\frac{\theta_1+\theta_2}{2}\right)<\mathrm{Re}\,\left(\frac{\theta_2-\theta_1}{2}\right)$, then rotating waves bifurcate as $c$ decreases through $c_*$
    \end{itemize}
The family of oblique stripe and checkerboard solutions can be parameterized in terms of the amplitude $a$ as
    \begin{align}
        u_{os} &=  u_*+2a\mathrm{Re}\,(e^{i(\tau+y)}p(x))+O(a^2),\\
        u_{cb}&=u_*+4a\cos(y)\mathrm{Re}\,(e^{i\tau}p(x))+O(a^2),
    \end{align}
and the bifurcation parameter $c$ can also be parameterized in terms of the amplitude $a$ for both oblique stripes and checkerboards as
\begin{align}
        c_{os}&=c_*-\frac{\mathrm{Re}\,(\theta_1)}{\mu'(c_*)}a^2+O(|a|^3),\label{c_os}\\
        c_{cb}&=c_*-\frac{\mathrm{Re}\,(\theta_1+\theta_2)}{\mu'(c_*)}a^2+O(|a|^3).\label{c_cb}
    \end{align}
\end{Theorem}

This theorem tells us when patterns bifurcate from our front solution $u_*$, as well as giving computable coefficients which determine the direction of bifurcation of patterns.

We prove this theorem in Section \ref{s:ab}.  In Section \ref{s:ex} we provide an example of a specific nonlinearity and heterogeneity for which we observe this behavior.  In Section \ref{s:disc}  we discuss our work and mention a few open areas of research stemming from it.  Finally, in the appendices we provide proofs establishing Fredholm properties of the linearization we consider, and give a short summary of the abstract theory of Hopf bifurcation with $O(2)$ symmetry.

\section{Abstract Results}\label{s:ab}
In this section, we seek to prove our theorem.  We approach \eqref{comoving_inhom} as an abstract nonlinear equation, for which the primary front $u_*$ is a zero. We first establish Fredholm properties of the linearization at $u_*$ in a space of $y,t$ periodic functions. The independence of the front $u_*$ in $y$ and $t$ allows for a Fourier decomposition of the operator and its Fredholm index.  We use exponentially weighted spaces to push neutral continuous spectrum away from the imaginary axis in the $t,y$-independent component. Then, using a co-domain restriction which projects off of the constants to address neutral mass flux in $x$, we obtain an operator which has Fredholm index 0, with kernel spanned by time-modulated forms of the transverse eigenfunctions. We then perform a Lyapunov-Schmidt decomposition to reduce the infinite-dimensional equation to a finite dimensional bifurcation equation in terms of the kernel variables, which can be put into $O(2)$-Hopf normal form, allowing us to establish bifurcating transversely patterned solutions and give expressions for the direction of bifurcation.  To begin, we introduce some useful notation.

\paragraph{Function spaces}

Recall, we consider the domain $(x,y,\tau)\in\mathbb{R}\times\mathbb{T}_y\times\mathbb{T}_\tau$, where $\mathbb{T}_{y,\tau}=[0,2\pi)$. We let
\[X:=L^2(\mathbb{T}_\tau),\qquad Y:=H^1(\mathbb{T}_\tau).\]
We then define the exponentially weighted $L^2$-space
\begin{equation}
\mathcal{X}:=L_\eta^2(\mathbb{R}\times\mathbb{T}_y,X)=\{v:\mathbb{R}\times\mathbb{T}_y\rightarrow X\,\,|\,\,\|v\|^2_{2,\eta}<\infty\}\nonumber
\end{equation}
with weighted norm 
\begin{equation}
    \|v\|^2_{2,\eta}:=\int_{\mathbb{R}\times \mathbb{T}_y}\| e^{\eta\langle x\rangle}v(x,y,\cdot)\|_{X}^2dxdy,\qquad  \langle x\rangle=\sqrt{1+x^2}.
\end{equation}
Given the following inner product, $\mathcal{X}$ becomes a Hilbert space:
\[\langle u,v\rangle_\mathcal{X}:=\frac{1}{4\pi^2}\int_0^{2\pi}\int_0^{2\pi}\int_{-\infty}^\infty u(x,y,\tau)\overline{v(x,y,\tau)}e^{2\eta\langle x\rangle}dxdy d\tau.\]

We can similarly define Sobolev spaces $H_\eta^k$ as 
\begin{equation}
    H_\eta^k(\mathbb{R}\times\mathbb{T},X)=\{v:\mathbb{R}\times\mathbb{T}\rightarrow X\,\,|\,\,\|D^\alpha v\|^2_{2,\eta}<\infty,\,\, |\alpha|\leq k\}.\nonumber
\end{equation}
Finally, we define the Banach space 
\[
\mathcal{Y}:=L^2_\eta(\mathbb{R}\times\mathbb{T}_y,Y)\cap H_\eta^4(\mathbb{R}\times\mathbb{T}_y,X),
\]
with norm \begin{equation}\|u\|^2_\mathcal{Y}:=\int_0^{2\pi}\int_{-\infty}^\infty \|u(x,y,\cdot)\|_Y^2+\sum_{|\alpha|\leq4}\|D^\alpha u(x,y,\cdot)\|^2_Xdxdy.\nonumber\end{equation}

\paragraph{Abstract nonlinear equation}

We consider perturbations  $u=u_*+v$ of the front solution $u_*(x)$ in \eqref{comoving_inhom} at the parameters $(\omega,c) = (\omega_*,c_*)$.   Defining $\Tilde{\omega}=\omega-\omega_*,\, \, \Tilde{c}=c-c_*$, and $\Omega=(\Tilde{\omega},\Tilde{c})$,  and subtracting off $v$-independent parts, we obtain: 
\begin{equation}
0 = (\tilde \omega + \omega_*)\partial_\tau v+\Delta_k(\Delta_k v+g(x,v;u_*))-(\tilde c + c_*)\partial_xv =: \mathcal{F}(v; \Omega)
\end{equation}
where $g(x,v;u_*):=f(x,u_*+v)-f(x,u_*)$. This defines a locally smooth mapping $\mathcal{F}:\mathcal{Y}\times\mathbb{R}^2\to\mathcal{X}$ with $(0;0,0)$ corresponding to the base front solution $u_*$ and general zeros $(v;\Omega)$ of $\mathcal{F}$ corresponding to  $y,\tau$-periodic solutions of \eqref{comoving_inhom}. The smoothness of $\mathcal{F}$ in the $v$ variable is dependent on the smoothness of $f$ in the same variable.  By Hypothesis \ref{hyp1} we have that $f$ is smooth in $v$, and hence so is $\mathcal{F}$. 

\subsection{Linear Properties}
Linearizing $\mathcal{F}$ at the front solution $(v;\tilde{\omega},\tilde{c})=(0;0,0)$, we obtain the following closed and densely defined linear operator
\begin{align}
\mathcal{L}:\mathcal{Y}\subset\mathcal{X}&\to\mathcal{X},\nonumber\\
v&\mapsto\omega_*\partial_\tau v+\Delta_k(\Delta_k v+\partial_uf(x,u_*)v)-c_*\partial_xv.\label{linearized}
\end{align}
We can further define the (formal) $L^2$-adjoint to be
\[\mathcal{L}^*=-\omega_*\partial_\tau+(\Delta_k+\partial_uf(x,u_*))\Delta_k+c_*\partial_x.\]
By restricting its co-domain, the operator $\mathcal{L}$ has the following Fredholm properties.
\begin{Proposition}\label{pfred}
    Let $\mathring{\mathcal{X}}:=\{u\in\mathcal{X}|\langle u,e^{-2\eta\langle x\rangle}\rangle_\mathcal{X}=0\}$.  Then for $\eta>0$ small, $\mathcal{L}:\mathcal{Y}\to\mathring{\mathcal{X}}$ is Fredholm of index 0, with four dimensional kernel.
\end{Proposition}
We leave the proof of this proposition to Appendix \ref{app1}.  The general scheme is to decompose the space $\mathcal{X}$ into various Fourier subspaces, study the Fredholm properties on each, and then use Fredholm algebra to determine the index of the full operator. We find on all but one subspace that $\mathcal{L}$ is a Fredholm operator with Fredholm index 0, while on the $t,y$-independent subspace $\mathcal{L}$ has index $-1$ when considered as a mapping into $\mathcal{X}$.  Restricting the codomain to $\mathring{\mathcal{X}}$, which projects off constants, then yields an index 0 operator. We find the kernel of $\mathcal{L}$ lies in the subspaces spanned by the base modes of the form $e^{\pm iy}e^{\pm i \tau}$.

As we wish to perform a Lyapunov-Schmidt reduction on $\mathcal{F}$, the first step is to decompose the domain and codomain of $\mathcal{L}$. Recall that $e^{iy}p(x), e^{-iy}\overline{p}(x)$, along with their $y$-reflections, give Hopf eigenfunctions of the linear operator $Lv=-\Delta_k(\Delta_k v+\partial_uf(x,u_*)v)+c_*\partial_xv$.  We then define 
\begin{align}
&P_+=e^{i(\tau+y)}p(x),\quad P_-=\overline{P_+},\nonumber\\
&Q_+=e^{i(\tau-y)}p(x),\quad Q_-=\overline{Q_+}.\nonumber
\end{align}
Then, by Hypotheses \ref{hyp3} and \ref{hyp5}, $\ker\mathcal{L}=span\{P_+,P_-,Q_+,Q_-\}$, and hence any $u_0\in\ker\mathcal{L}$ is given by \[u_0=aP_++\overline{a}P_-+bQ_++\overline{b}Q_-,\]
with $a,b\in\mathbb{C}$.  We similarly have elements of the $L^2$-adjoint kernel: 
$$
\Psi_+=e^{i(\tau+y)}\psi_+,\quad \Psi_- = \overline{\Psi_+},\qquad \Phi_+=e^{i(\tau-y)}\psi_+,\quad \Phi_- = \overline{\Phi_+}.
$$
  Then we can decompose the domain and codomain of $\mathcal{L}$ as 
  \[
  \mathcal{Y}=\ker\mathcal{L}\oplus M \qquad \mathring{\mathcal{X}}=N\oplus\ker\mathcal{L}^*,
  \]
where $N=(\ker\mathcal{L}^*)^\perp$ and $M=(\ker\mathcal{L})^\perp$.  We then define projections $E:\mathring{\mathcal{X}}\to\ker\mathcal{L}^*$ and $1-E$ with 
\[
E\mathcal{F}=\sum_{i=+,-}\langle\mathcal{F},\Psi_i\rangle_\mathcal{X}\cdot\Psi_i+\langle\mathcal{F},\Phi_i\rangle_\mathcal{X}\cdot\Phi_i.
\]
Finally, for $u_0\in\ker\mathcal{L}$ and $u_h\in M$, we obtain the decomposed system of equations
\begin{align}
E\mathcal{F}(u_0+u_h;\Omega)&=0\label{e:bif}\\
(1-E)\mathcal{F}(u_0+u_h;\Omega)&=0.
\end{align}

 Since $(1-E)\mathcal{F}$ has invertible linearization in $u_h$ at $u_0=0$, the implicit function theorem gives that there exists a smooth mapping $w:\ker\mathcal{L}\times\mathbb{R}^2\to M$ satisfying \[(1-E)\mathcal{F}(u_0+w(u_0;\Omega);\Omega)=0.\]
Furthermore, this solution satisfies $w(0;\Omega)=0.$ We also have
\begin{align}
 0 = \frac{\partial\mathcal{F}}{\partial u_0}\Big|_{(0;0,0)}&=\omega_*\partial_\tau\frac{\partial w}{\partial u_0}+\Delta_k\left(\Delta_k\frac{\partial w}{\partial u_0}+\partial_uf(x,u_*)\frac{\partial w}{\partial u_0}\right)-c_*\partial_x\frac{\partial w}{\partial u_0}=\mathcal{L}\frac{\partial w}{\partial u_0}\nonumber 
\end{align}
and so we have that $\frac{\partial w}{\partial u_0}\in\ker\mathcal{L}$, while $w\in(\ker\mathcal{L})^\perp$.  Hence $\frac{\partial w}{\partial u_0}(0;0,0)$ is both tangent and normal to $w$, and so $\frac{\partial w}{\partial u_0}(0;0,0)=0$.  So we can expand $w$ in terms of the kernel coordinates $a$ and $b$ as \[w(a,\overline{a},b,\overline{b};\Omega)=a^2e^{2i(\tau+y)}\phi_{aa}(x)+|a|^2\phi_{a\overline{a}}(x)+abe^{2i\tau}\phi_{ab}(x)+a\overline{b}e^{2iy}\phi_{a\overline{b}}(x)+\]
\[\overline{a}^2e^{-2i(\tau+y)}\phi_{\overline{aa}}(x)+\overline{a}be^{-2iy}\phi_{\overline{a}b}(x)+\overline{a}\overline{b}e^{-2i\tau}\phi_{\overline{a}\overline{b}}(x)+b^2e^{2i(\tau-y)}\phi_{bb}(x)+|b|^2\phi_{b\overline{b}}(x)+\]
\[\overline{b}^2e^{-2i(\tau-y)}\phi_{\overline{bb}}(x)+\mathcal{O}(u_0^3)\]

Here $\phi_{ij}\in M\subset\mathcal{Y}$ for all $i,j\in\{a,\overline{a},b,\overline{b}\}$.  Rearranging $(1-E)\mathcal{F}=0$, we obtain
\begin{align}
&\mathcal{L}(u_0+w(u_0;\Omega))=-\Delta_k(f(x,u_*+u_0+w(u_0;\Omega))-f(x,u_*)-\partial_uf(x,u_*)(u_0+w(u_0;\Omega)))\nonumber\\
 &=-\Delta_k\left(\frac{1}{2}\partial_u^2f(x,u_*)(u_0+w(u_0;\Omega))^2+\frac{1}{6}\partial_u^3f(x,u_*)(u_0+w(u_0;\Omega))^3+\mathcal{O}(\|u_0\|_\mathcal{Y}^4)\right).
\end{align}
Substituting in the above expansion and projecting onto the different Fourier modes $e^{2i(j\tau+\ell y)}$ with $j,\ell\in\{0,\pm1\}$, we find the following system of equations at quadratic order in $u_0$, after dividing out constants:
\begin{align}
    \mathcal{L}(e^{2i(\tau+y)}\phi_{aa})&=-\Delta_k\left(\frac{1}{2}\partial_u^2f(x,u_*)e^{2i(\tau+y)}p^2(x)\right),\label{phi_a_a}\\
    \mathcal{L}(\phi_{a\overline{a}})&=-\Delta_k\left(\frac{1}{2}\partial_u^2f(x,u_*)|p(x)|^2\right),\label{phi_a_bara}\\
    \mathcal{L}(e^{2i\tau}\phi_{ab})&=-\Delta_k\left(\frac{1}{2}\partial_u^2f(x,u_*)e^{2i\tau}p^2(x)\right),\label{phi_a_b}\\
    \mathcal{L}(e^{2iy}\phi_{a\overline{b}})&=-\Delta_k\left(\frac{1}{2}\partial_u^2f(x,u_*)e^{2iy}|p(x)|^2\right),\label{phi_a_barb}\\
    \mathcal{L}(e^{-2i(\tau+y)}\phi_{\overline{a}\overline{a}})&=-\Delta_k\left(\frac{1}{2}\partial_u^2f(x,u_*)e^{-2i(\tau+y)}\overline{p}^2(x)\right),\label{phi_bara_bara}\\
    \mathcal{L}(e^{-2iy}\phi_{\overline{a}b})&=-\Delta_k\left(\frac{1}{2}\partial_u^2f(x,u_*)e^{-2iy}|p(x)|^2\right),\label{phi_bara_b}\\
    \mathcal{L}(e^{-2i\tau}\phi_{\overline{a}\overline{b}})&=-\Delta_k\left(\frac{1}{2}\partial_u^2f(x,u_*)e^{-2i\tau}\overline{p}^2(x)\right),\label{phi_bara_barb}\\
    \mathcal{L}(e^{2i(\tau-y)}\phi_{bb})&=-\Delta_k\left(\frac{1}{2}\partial_u^2f(x,u_*)e^{2i(\tau-y)}p^2(x)\right),\label{phi_b_b}\\
    \mathcal{L}(\phi_{b\overline{b}})&=-\Delta_k\left(\frac{1}{2}\partial_u^2f(x,u_*)|p(x)|^2\right),\label{phi_b_barb}\\
    \mathcal{L}(e^{-2i(\tau-y)}\phi_{\overline{b}\overline{b}})&=-\Delta_k\left(\frac{1}{2}\partial_u^2f(x,u_*)e^{-2i(\tau-y)}\overline{p}^2(x)\right).\label{phi_barb_barb}
\end{align}

\begin{Lemma}
    Equations \eqref{phi_a_a} to \eqref{phi_barb_barb} can be uniquely solved for the functions $\phi_{ij}\in M$, $i,j\in\{a,\overline{a},b,\overline{b}\}$.
\end{Lemma}
\begin{proof}
First, the right hand side in each of \eqref{phi_a_a} - \eqref{phi_barb_barb} is contained in $\mathring{\mathcal{X}}$. We note each right-hand side takes the form $\Delta_k(H)$ for some function $H$. As $p(x)$ is exponentially localized, so is $H$. Hence, integration by parts gives that $\langle \Delta_k H, e^{-2\eta\langle x \rangle} \rangle_{L^2_\eta} = \langle \Delta_k H, 1\rangle_{L^2} = 0.$ 

Next, we claim that each right hand side of \eqref{phi_a_a} - \eqref{phi_barb_barb} is contained in the range of $\mathcal{L}$.  This is seen by observing that $\ker \mathcal{L}^*$ is spanned by the functions $\Psi_+,\Psi_-,\Phi_+,$ and $\Phi_-$, all of which have $\tau$-dependence of the form $e^{\pm i\tau}$, while the right hand sides of \eqref{phi_a_a} - \eqref{phi_barb_barb} contain $\tau$ dependence of the form $e^{0i\tau}$ or $e^{\pm 2i\tau}$ and hence must lie in $(\ker\mathcal{L}^*)^\perp$. Uniqueness follows from the fact that each $\phi_{ij}\in M=(\ker\mathcal{L})^\perp$. 

\end{proof}

Thus we can solve equations \eqref{phi_a_a}-\eqref{phi_barb_barb} for $\phi_{ij}$, which in turn will give the leading-order expansion of $w(u_0;\Omega)$.  
Plugging this into the bifurcation equation \eqref{e:bif}  we obtain 
\[0=E\mathcal{F}(u_0+w(u_0;\Omega);\Omega) = \sum_{i=+,-}\langle\mathcal{F}(u_0+w(u_0;\Omega);\Omega),\Psi_i\rangle_\mathcal{X}\cdot\Psi_i+\langle\mathcal{F}(u_0+w(u_0;\Omega);\Omega),\Phi_i\rangle_\mathcal{X}\cdot\Phi_i,\]
which, since $\Psi_+,\Psi_-,\Phi_+,$ and $\Phi_-$ are linearly independent, is equivalent to the finite dimensional system of equations
 \begin{align}
    0=\Theta_1(u_0;\Omega):=\langle\mathcal{F}(u_0+w(u_0;\Omega);\Omega),\Psi_+\rangle_\mathcal{X}&\\
0=\Theta_2(u_0;\Omega):=\langle\mathcal{F}(u_0+w(u_0;\Omega);\Omega),\Psi_-\rangle_\mathcal{X}&\\
0=\Theta_3(u_0;\Omega):=\langle\mathcal{F}(u_0+w(u_0;\Omega);\Omega),\Phi_+\rangle_\mathcal{X}&\\
0=\Theta_4(u_0;\Omega):=\langle\mathcal{F}(u_0+w(u_0;\Omega);\Omega),\Phi_-\rangle_\mathcal{X}.&
\end{align}

Then by the definitions of $\Psi_i$ and $\Phi_i$ we have that the only terms which are nonzero in the inner product will be terms of the form $e^{i\ell_\tau\tau}e^{i\ell_yy}$ for $\ell_\tau,\ell_y=\pm1$.  Each of the functions $\Theta_i$ is zero trivially when $a=b=0$, and so we can decompose each equation as $\Theta_i = r_i(a,\overline{a},b,\overline{b};\Omega)\cdot h_i(a,\overline{a},b,\overline{b})$, where $h(0,0,0,0) = 0$. In the coordinates $(a,\bar a, b, \bar{b})\in \mathbb{C}^4$, the reduced equations are equivariant under the following actions induced by the symmetries of the original nonlinear equation: time translation induces the action $(a,b)\mapsto (e^{i \theta} a, e^{i \theta} b),\quad \theta\in[0,2\pi)$, $y$-translation induces $(a,b)\mapsto (e^{i \phi} a, e^{-i \phi} b), \phi\in[0,2\pi)$, while $y$-reflection induces $(a,b)\mapsto (b,a)$.  Under these symmetries, we readily conclude that $r_i$ can be written as a function of $|a|^2, |b|^2$ and $\Omega$.  Because of the inner product with the adjoint kernel elements $\Psi_+,\Psi_-,\Phi_+$, and $\Phi_-$, each $h_i$ will be linear in exactly one of the kernel coordinates, and will not depend on any of the others.  Thus, possibly after redefinition by a constant, we can write
\[\Theta_1=r_1(|a|^2,|b|^2;\Tilde{\omega},\Tilde{c})a\]
\[\Theta_2=r_2(|a|^2,|b|^2;\Tilde{\omega},\Tilde{c})\overline{a}\]
\[\Theta_3=r_3(|a|^2,|b|^2;\Tilde{\omega},\Tilde{c})b\]
\[\Theta_4=r_4(|a|^2,|b|^2;\Tilde{\omega},\Tilde{c})\overline{b}.\]
We note that expansions of each $r_i$ are determined by taking derivatives of $\mathcal{F}$ and taking inner products with the appropriate $\Psi_i$ or $\Phi_i$.  We consider these in cases by the signs of $\ell_\tau$ and $\ell_y$, in the following way:

In the case $\ell_\tau = \ell_y = 1$, we must consider $\langle\mathcal{F},e^{i(\tau+y)}\psi_+\rangle_\mathcal{X}$, and so the only nonzero terms in the inner product will be terms with the mode $e^{i(\tau+y)}$ by the orthogonality of the exponential.

\textit{To find $\partial_{\Tilde{c}}r_1$:}  We consider $\mathcal{F}_{a\Tilde{c}}$, and we will then take the appropriate inner product, finding
\begin{align}
\mathcal{F}_{a\Tilde{c}}|_0&=\Bigg[(\Tilde{c}+c_*)\partial_x\Bigg(ae^{i(\tau+y)}p(x)+\overline{a}e^{-i(\tau+y)}\overline{p}(x)\nonumber\\
&\qquad\qquad +be^{i(\tau-y)}p(x)+\overline{b}e^{-i(\tau-y)}\overline{p}(x)+w(a,\overline{a},b,\overline{b};\Omega)\Bigg)\Bigg]_{a\Tilde{c}} = e^{i(\tau+y)}p'(x).
\end{align}
So $\partial_{\Tilde{c}}r_1(0)=\langle e^{i(\tau+y)}p'(x),e^{i(\tau+y)}\psi_+\rangle_\mathcal{X}=\lambda_{\Tilde{c}}(0).$

\textit{To find $\partial_{\Tilde{\omega}}r_1$:}  We consider $\mathcal{F}_{a\Tilde{\omega}}$ and take the appropriate inner product, finding
\begin{align}
\mathcal{F}_{a\Tilde{\omega}}|_0&=\Bigg[(\Tilde{\omega}+\omega_*)\Bigg(ae^{i(\tau+y)}p(x)+\overline{a}e^{-i(\tau+y)}\overline{p}(x)+be^{i(\tau-y)}p(x)+\nonumber\\
&\qquad\qquad+\overline{b}e^{-i(\tau-y)}\overline{p}(x)+w(a,\overline{a},b,\overline{b};\Omega)\Bigg)_\tau\Bigg]_{a\Tilde{\omega}}=ie^{i(\tau+y)}p(x),
\end{align}
and so $\partial_{\Tilde{\omega}}r_1(0)=\langle ie^{i(\tau+y)}p(x),e^{i(\tau+y)}\psi_+\rangle_\mathcal{X}=i$.

\textit{To find $\partial_{|a|^2}r_1$:}  We must consider all the different ways one can produce nonzero terms of the correct mode, $e^{i(\tau+y)}$.  From the right-hand sides of the equations defining the $\phi_{ij}$'s, this can only occur from terms of the form $u_0^3$ or $u_0\cdot w(u_0;\Omega)$.  By considering these terms and taking the appropriate inner products, we will find that
\[\partial_{|a|^2}r_1=\left\langle-\Delta_k\left({3\choose 2,1}\frac{1}{6}\partial_u^3f(x,u_*)e^{i(\tau+y)}p^2(x)\overline{p}(x)+\right.\right.\]\[\left.\left.+{2\choose 1,1}\frac{1}{2}\partial_u^2f(x,u_*)e^{i(\tau+y)}p(x)\phi_{a\overline{a}}+{2\choose 1,1}\frac{1}{2}\partial_u^2f(x,u_*)e^{i(\tau+y)}\overline{p}(x)\phi_{aa}\right),e^{i(\tau+y)}\psi_+\right\rangle_\mathcal{X}=\]
\[=\left\langle-\Delta_k\left(e^{i(\tau+y)}\frac{1}{2}\partial_u^3f(x,u_*)p^2\overline{p}+e^{i(\tau+y)}\partial_u^2f(x,u_*)[p\phi_{a\overline{a}}+\overline{p}\phi_{aa}]\right),e^{i(\tau+y)}\psi_+\right\rangle_\mathcal{X}+h.o.t.\]

where ${N\choose n_1...n_k}=\frac{N!}{n_1!\cdots n_k!}$.

\textit{To find $\partial_{|b|^2}r_1$:}  Similar to how we found $\partial_{|a|^2}r_1$, we look for nonzero terms with the correct mode, which will only occur in the right-hand sides from $u_0^3$ or $u_0\cdot w(u_0;\Omega)$.  By taking the appropriate inner products of these terms and rewriting the Laplacian as $\Delta_k=\partial_x^2-k^2$ so that the term $e^{i(\tau+y)}$ commutes with the Laplacian, we find
\[\partial_{|b|^2}r_1=\left\langle-\Delta_k\left(\partial_u^3f(x,u_*)p^2\overline{p}+\partial_u^2f(x,u_*)[p\phi_{b\overline{b}}+p\phi_{a\overline{b}}+\overline{p}\phi_{ab}]\right),\psi_+\right\rangle_{L^2}+h.o.t.\]

Thus in total we have (with $\Delta_k=\partial_x^2-k^2$) 
\begin{align}
\Theta_1&=(\lambda_{\Tilde{c}}(0)\Tilde{c}+i\Tilde{\omega})a+a|a|^2\left\langle-\Delta_k\left(\frac{1}{2}\partial_u^3f(x,u_*)p^2\overline{p}+\partial_u^2f(x,u_*)[p\phi_{a\overline{a}}+\overline{p}\phi_{aa}]\right),\psi_+\right\rangle_{L^2}\nonumber\\
&+a|b|^2\left\langle-\Delta_k\left(\partial_u^3f(x,u_*)p^2\overline{p}+\partial_u^2f(x,u_*)[p\phi_{b\overline{b}}+p\phi_{a\overline{b}}+\overline{p}\phi_{ab}]\right),\psi_+\right\rangle_{L^2}+h.o.t.
\end{align}

Employing similar computations for the subspaces $\ell_\tau=\ell_y=-1$, $\ell_\tau=-\ell_y=1$, and $-\ell_\tau=\ell_y=1$ we respectively find
\begin{align}
\Theta_2&=\overline{\Theta_1} = (\overline{\lambda}_{\Tilde{c}}(0)\Tilde{c}-i\Tilde{\omega})\overline{a}
+\overline{a}|a|^2\left\langle-\Delta_k\left(\frac{1}{2}\partial_u^3f(x,u_*)\overline{p}^2p+\partial_u^2f(x,u_*)[\overline{p}\phi_{a\overline{a}}+p\phi_{\overline{a}\overline{a}}]\right),\psi_-\right\rangle_{L^2}\nonumber\\
&+\overline{a}|b|^2\left\langle-\Delta_k\left(\partial_u^3f(x,u_*)\overline{p}^2p+\partial_u^2f(x,u_*)[\overline{p}\phi_{b\overline{b}}+p\phi_{\overline{a}\overline{b}}+\overline{p}\phi_{\overline{a}b}]\right),\psi_-\right\rangle_{L^2}+h.o.t.\\
\Theta_3&=(\lambda_{\Tilde{c}}(0)\Tilde{c}+i\Tilde{\omega})b
+b|a|^2\left\langle-\Delta_k\left(\partial_u^3f(x,u_*)p^2\overline{p}+\partial_u^2f(x,u_*)[p\phi_{a\overline{a}}+p\phi_{\overline{a}b}+\overline{p}\phi_{ab}]\right),\psi_+\right\rangle_{L^2}\nonumber\\
&+b|b|^2\left\langle-\Delta_k\left(\frac{1}{2}\partial_u^3f(x,u_*)p^2\overline{p}+\partial_u^2f(x,u_*)[p\phi_{b\overline{b}}+\overline{p}\phi_{bb}]\right),\psi_+\right\rangle_{L^2}+h.o.t.\label{e:TH3}\\
\Theta_4&=\overline{\Theta_3} = (\overline{\lambda}_{\Tilde{c}}(0)\Tilde{c}-i\Tilde{\omega})\overline{b}
+\overline{b}|a|^2\left\langle-\Delta_k\left(\partial_u^3f(x,u_*)\overline{p}^2p+\partial_u^2f(x,u_*)[p\phi_{a\overline{a}}+p\phi_{\overline{a}\overline{b}}+\overline{p}\phi_{a\overline{b}}]\right),\psi_+\right\rangle_{L^2}\nonumber\\
&+\overline{b}|b|^2\left\langle-\Delta_k\left(\frac{1}{2}\partial_u^3f(x,u_*)\overline{p}^2p+\partial_u^2f(x,u_*)[p\phi_{\overline{b}\overline{b}}+\overline{p}\phi_{b\overline{b}}]\right),\psi_+\right\rangle_{L^2}+h.o.t.
\end{align}

\subsection{Reduced Equations}
Following \cite[Ch. 17, Prop. 2.1]{Golubitsky}, we suppress the complex conjugate equations and put our bifurcation equations in the form\begin{equation}
    \begin{pmatrix}
        \Theta_1(a,b)\\\Theta_3(a,b)
    \end{pmatrix}=(p+iq)\begin{pmatrix}
        a\\b
    \end{pmatrix}+(r+is)\delta\begin{pmatrix}
        a\\-b
    \end{pmatrix},\label{eq:normal_form}
\end{equation}
where $p,q,r,s$ are polynomials in the variables $N=|a|^2+|b|^2$ and $D=\delta^2=(|b|^2-|a|^2)^2$.  This will allow us to determine the direction of the bifurcation for both the standing and rotating waves, whether supercritical or subcritical.  We have
\[\Theta_1(a,b)=(\lambda_{\Tilde{c}}(0)\Tilde{c}+i\Tilde{\omega})a+\theta_1 a|a|^2+\theta_2 a |b|^2+\mathcal{O}(|u_0|^4+|u_0|^2\cdot |w|+|w|^2),\]
with 
\begin{align}
    \theta_1=\left\langle-\Delta_k\left(\frac{1}{2}\partial_u^3f(x,u_*)p^2\overline{p}+\partial_u^2f(x,u_*)[p\phi_{a\overline{a}}+\overline{p}\phi_{aa}]\right),\psi_+\right\rangle_{L^2},\\
    \theta_2=\langle-\Delta_k(\partial_u^3f(x,u_*)p^2\overline{p}+\partial_u^2f(x,u_*)[p\phi_{b\overline{b}}+p\phi_{a\overline{b}}+\overline{p}\phi_{ab}]),\psi_+\rangle_{L^2}.
\end{align}

  Since we can write $|a|^2=\frac{1}{2}N-\frac{1}{2}\delta$ and $|b|^2=\frac{1}{2}N+\frac{1}{2}\delta$, this gives us the leading order form \[\Theta_1\approx a\left[\lambda_{\Tilde{c}}(0)\Tilde{c}+i\Tilde{\omega}+\frac{\theta_1+\theta_2}{2}N+\frac{\theta_2-\theta_1}{2}\delta\right].\]

Using this, we can deduce the leading order forms of the polynomials $p,q,r,$ and $s$ as
\begin{align}
    p\approx\mu_{\Tilde{c}}(0)\Tilde{c}+\mathrm{Re}\,\left(\frac{\theta_1+\theta_2}{2}\right)N,\\
    q\approx\Tilde{\omega}+\kappa_{\Tilde{c}}(0)\Tilde{c}+\mathrm{Im}\left(\frac{\theta_1+\theta_2}{2}\right)N,\\
    r\approx\mathrm{Re}\,\left(\frac{\theta_2-\theta_1}{2}\right),\\
    s\approx\mathrm{Im}\left(\frac{\theta_2-\theta_1}{2}\right).
\end{align}

However this selection of $p,q,r,$ and $s$ must also hold for $\Theta_3$.  $\Theta_3$ can be written similarly to $\Theta_1$ as \[\Theta_3\approx b\left[\lambda_{\Tilde{c}}(0)\Tilde{c}+i\Tilde{\omega}+\eta_1|a|^2+\eta_2|b|^2\right]=\]
\[=b\left[\lambda_{\Tilde{c}}(0)\Tilde{c}+i\Tilde{\omega}+\frac{\eta_1+\eta_2}{2}N+\frac{\eta_1-\eta_2}{2}\delta\right],\]
with $\eta_1$ and $\eta_2$ being the appropriate inner products from $\Theta_3$ in \eqref{e:TH3}. We can perform similar assignments as above to get forms for $p,q,r,$ and $s$ in terms of $\eta_1$ and $\eta_2$.

We claim that $\phi_{a\overline{a}}=\phi_{b\overline{b}}$, $\phi_{a\overline{b}}=\phi_{\overline{a}b}$, and $\phi_{aa}=\phi_{bb}$.  Then it follows that $\eta_2=\theta_1$ and $\eta_1 = \theta_2$ and hence that $\theta_1+\theta_2=\eta_1+\eta_2$ and $\theta_2-\theta_1=\eta_1-\eta_2$, which in turn shows that $\begin{pmatrix}\Theta_1\\\Theta_3\end{pmatrix}$ can be written in the form \eqref{eq:normal_form}.  We show the claimed equalities next.

\subsubsection{$\phi_{a\overline{a}}=\phi_{b\overline{b}}$}

Returning to equations \eqref{phi_a_bara} and \eqref{phi_b_barb}, we note that $\phi_{a\overline{a}}$ and $\phi_{b\overline{b}}$ solve the same equation and hence are equal since the equations can be solved uniquely.

\subsubsection{$\phi_{a\overline{b}}=\phi_{\overline{a}b}$}

Here we compare \eqref{phi_a_barb} and \eqref{phi_bara_b}. Recall we have $\mathcal{L}v=\omega_*\partial_\tau v+\Delta_k(\Delta_k v+\partial_uf(x,u_*)v)-c_*\partial_xv$.  From \eqref{phi_a_barb} we have
\[\Delta_k(\Delta_ke^{2iy}\phi_{a\overline{b}}+\partial_uf(x,u_*)e^{2iy}\phi_{a\overline{b}})-c_*\partial_xe^{2iy}\phi_{a\overline{b}}=\]\[=-\Delta_k\left(\frac{1}{2}\partial_u^2f(x,u_*)e^{2iy}|p(x)|^2\right).\]

We then note that because we have functions $e^{2iy}\phi_{a\overline{b}}$ and $e^{2iy}|p(x)|^2$, the Laplacian becomes $\Delta_k=\partial_x^2+k^2\partial_y^2=\partial_x^2-4k^2$.  Dividing by $e^{2iy}$, this leaves us with
\[(\partial_x^2-4k^2)((\partial_x^2-4k^2)\phi_{a\overline{b}}+\partial_uf(x,u_*)\phi_{a\overline{b}})-c_*\partial_x\phi_{a\overline{b}}=\]
\[=\frac{1}{2}(4k^2-\partial_x^2)(\partial_u^2f(x,u_*)|p(x)|^2).\]

For \eqref{phi_bara_b}, note that again the Laplacian becomes $\Delta_k=\partial_x^2-4k^2$ because of the form of the functions.  Evaluating the $y$ dependence and dividing by $e^{-2iy}$, we find:
\[(\partial_x^2-4k^2)((\partial_x^2-4k^2)\phi_{\overline{a}b}+\partial_uf(x,u_*)\phi_{\overline{a}b})-c_*\partial_x\phi_{\overline{a}b}=\]
\[=-\frac{1}{2}(\partial_x^2-4k^2)(\partial_u^2f(x,u_*)|p(x)|^2)=\frac{1}{2}(4k^2-\partial_x^2)(\partial_u^2f(x,u_*)|p(x)|^2).\]

then we can see that $\phi_{a\overline{b}}$ and $\phi_{\overline{a}b}$ solve the same equation.  Thus they are equal by the same reasoning as before.

\subsubsection{$\phi_{aa}=\phi_{bb}$}

This follows a similar procedure as the previous case, though now there is dependence on $\tau$.  We can show that $\phi_{aa}$ and $\phi_{bb}$ solve the same equation, and thus must be equal.  Thus we have shown that $\eta_2=\theta_1$ and $\theta_2=\eta_1$, and so we can assign the polynomials $p,q,r,$ and $s$ as desired.

\subsection{Bifurcations}

With our bifurcation equation in the desired form, we can now conclude the existence of the desired solutions and determine the bifurcation directions from the front solution $u_*$.  The results of  \cite[Ch. 17]{Golubitsky} readily give the existence of the pair of bifurcating solution branches, solving for $\Omega$ as a function of the solution amplitude. Furthermore, all that is required to determine the direction of bifurcation is $\frac{\partial p}{\partial N}(0)$ and $r(0)$, where $p$ and $r$ are the polynomials from the normal form.  Straightforward calculation gives
 \[
 p_N(0)=\mathrm{Re}\,\left(\frac{\theta_1+\theta_2}{2}\right)=:\alpha,\qquad 
 r(0)=\mathrm{Re}\,\left(\frac{\theta_2-\theta_1}{2}\right)=:\beta.
 \]

Using these, we can then determine the direction of bifurcation of the two branches of solutions.  If $\alpha<0$ and $\beta>0$, then both solution branches bifurcate as $c$ moves below $c_*$, which we refer to as a \textbf{type 1} bifurcation.  In contrast, if $\alpha>0$ and $\beta<0$, then both branches bifurcate to the right, which we will call a \textbf{type 3} bifurcation.

If $\alpha,\beta>0$, then we must consider two cases: $\alpha<\beta$ and $\alpha>\beta$.  If $\alpha>\beta$, then both branches bifurcate to the right, a \textbf{type 3} bifurcation.  If $\alpha<\beta$, then rotating waves, corresponding to oblique stripes, bifurcate to the left, and standing waves, corresponding to checkerboard patterns, to the right, a \textbf{type 2} bifurcation.

If $\alpha,\beta<0$, then we again must consider two cases: $\alpha<\beta$ and $\alpha>\beta$.  If $\alpha<\beta$, then both branches bifurcate to the left, a \textbf{type 1} bifurcation.  If $\alpha>\beta$, the standing waves bifurcate to the left, and rotating waves to the right, a \textbf{type 4} bifurcation.

\begin{figure}[htbp]
    \centering
    \includegraphics[scale=0.6]{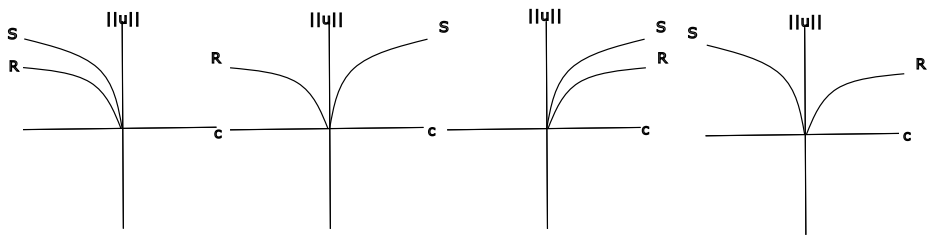}
    \caption{Here we see examples of all the different bifurcations possible: Type 1 (left), Type 2 (second), Type 3 (third), and Type 4 (right).  S represents the branch of standing waves (checkerboard patterns), and R denotes the rotating waves (oblique stripe patterns).  Here $c$ is our bifurcation parameter, the quench speed, and $\|u\|$ is an $L^2$ norm.}
    \label{fig:bif}
\end{figure}

\subsection{Leading Order Forms of the Solutions}
Using the decomposition $v = u_0 + w$, and the fact that $w$ is higher order, we obtain the following expansion for the full patterned front solution
\[u= u_*+(aP_++\overline{a}P_-+bQ_++\overline{b}Q_-)+O(|a|^2+|b|^2).\]

Following \cite[Ch. 17, Sec. 3]{Golubitsky}, in order to find zeros of the normal form \eqref{eq:normal_form} it suffices to take the real representatives from the $O(2)\times S^1$-orbit of solutions and thus restrict to $a,b\in \mathbb{R}$. 
Here rotating waves correspond to $a>b=0$, arising from the isotropy subgroup $\widetilde{SO}(2)=\{(\theta,\theta)\,\,|\,\,\theta\in S^1\}$, while standing waves correspond to $a=b>0$, arising from the isotropy subgroup $\mathbb{Z}_2\oplus\mathbb{Z}_2^c$. Using this information, we can determine the particular leading order form for both the rotating and standing waves, again which correspond to oblique and  checkerboard patterns.  This gives the solution forms
\begin{align}
u_{os}&=u_*+a(P_++P_-)+O(|a|^2)\nonumber\\
&=u_*+2a\mathrm{Re}\,\left(e^{i(\tau+y)}p(x)\right)+O(|a|^2),\\
u_{cb}&=u_*+\left(aP_+ + aP_-+aQ_+ +aQ_-\right)+O(|a|^2)\nonumber\\
&=u_*+4a\cos(y)\mathrm{Re}\,\left(e^{i\tau}p(x)\right)+O(|a|^2).
\end{align}
The expansions for the bifurcation parameter in terms of the amplitude, Equations \eqref{c_os} and \eqref{c_cb}, come from \cite[Ch. 17, Sec. 3]{Golubitsky}.  This concludes the proof of Theorem \ref{thm1}.

\section{Example}\label{s:ex}
In this section, we lay out an explicit example of a nonlinearity and heterogeneity which give rise to such bifurcations.  We will show using both rigorous and numerical evidence that these satisfy our hypotheses in Section \ref{sec3.1}, and we numerically determine the direction of the bifurcation in Section \ref{sec3.2}.  In Section \ref{sec3.3} we examine what happens as we vary the transverse wavenumber $k$.

\subsection{Nonlinearity and Heterogeneity}\label{sec3.1}
We fix $k=\frac{1}{2}$ and define the nonlinearity 
\begin{equation}\label{e:fex}
	f(x-ct,u)=h(x-ct)u+\gamma u^3-u^5,
\end{equation}
with top-hat heterogeneity 
\begin{equation}\label{e:hex}
h(\Tilde{x})=\tanh(\delta(\Tilde{x}-K))\tanh(-\delta(\Tilde{x}+K)),
\end{equation}
and $\delta\gg1$ and $K$ large.  In our numerics, described below, we set $\delta=5$ and $K=10\pi$. To understand this nonlinearity, we briefly discuss free invasion fronts in the unquenched, homogeneous coefficient nonlinearity \begin{equation}\Tilde{f}(u)=u+\gamma u^3-u^5.\end{equation} 
In the regime $\gamma<0$, invasion fronts into the unstable state $u\equiv0$ are determined by linear information of the state ahead of the front, and are known as \emph{pulled} fronts. In the regime $\gamma>1$, fronts are governed by the strong nonlinear growth and travel faster than the linear information ahead of the front predicts, commonly referred to as \emph{pushed} fronts. An example of such behavior can be found in \cite{goh2016pattern}, and more detail can be found in \cite[Sec. 2.6]{van2003front}.

We now verify that equation \eqref{comoving_inhom} with nonlinearity \eqref{e:fex} satisfies our hypotheses.  Clearly $f$ is smooth in both $x$ and $u$ and, as $x\to\pm\infty$, $f_\pm(u)=-u+\gamma u^3-u^5$ with the appropriate convergence rate.  Thus Hypothesis \ref{hyp1} is satisfied.

For Hypothesis \ref{hyp2}, our primary front will be the trivial solution $u_*\equiv0$.

For the spectral assumptions, Hypotheses \ref{hyp3} - \ref{hyp6}, we first note that essential spectrum of the linearized operator
\[Lv:=-\Delta_k(\Delta_kv+\partial_uf(x,u_*)v)+c_*\partial_xv\]
can be calculated explicitly.  We insert the ansatz $u(x,y,\tau)=e^{\lambda\tau+\nu x+i\ell y}$ into the linearized equation, $v_t = L v$ with $h \equiv -1$, to obtain the dispersion relation
\begin{equation} 
d(\lambda,\nu;k,\ell,c)=-(\nu^2-k^2\ell^2)[(\nu^2-k^2\ell^2)-1]+c\nu-\lambda.\label{dispersion}
\end{equation} 
 To determine the essential spectrum (on a space which is not exponentially weighted), we solve $d(\lambda,im)=0$ for $\lambda$ in terms of $m$ and all parameters, obtaining
 \begin{equation}
 \Sigma_{ess}=\{-(-m^2-k^2\ell^2)[(-m^2-k^2\ell^2)-1]+icm \,\,|\,\, m\in\mathbb{R}\}.\label{ess_spec}
 \end{equation}
Computation of such curves readily gives that the essential spectrum is contained in the left half-plane, and only touches the imaginary axis in a quadratic tangency at the origin.  Furthermore,  one can readily calculate that the real part of the numerical point spectrum corresponding to $\Sigma_{ess}$ also does not vary with $c$.  Hence there is no resonant essential spectrum away from $\lambda = 0$ for $c$ near the Hopf point.

We study point spectrum numerically.  We truncate the linear operator 
\[
Lv:=-\Delta_k(\Delta_kv+\partial_uf(x,u_*)v)+c_*\partial_xv,
\]
to a bounded computational domain $(x,y)\in [-M,M]\times [0,2\pi)$ with periodic boundary conditions, discretize it spectrally, and evaluate it using the Fast Fourier Transform. Approximate eigenvalues and eigenfunctions are then calculated using the MATLAB command `eigs'. See Figure \ref{fig:spec_at_bif} for a depiction of the numerical spectrum near the bifurcation speed. Here, as the system possesses periodic boundary conditions, the work \cite{sandstede2000absolute} implies that the point spectrum of the truncated operator accumulate onto the essential and point spectrum of the unbounded domain operator as $M\rightarrow+\infty$. We thus conjugate the operator with an exponential weight which pushes the essential spectrum $\Sigma_{ess}$ into the left half-plane, leaving only numerical eigenvalues which approximate the point spectrum of the unbounded domain operator. In the exponentially weighted case, there are no eigenvalues at 0 due to the lack of translation symmetry in $x$ and that the front $u_*$ is trivial. This gives Hypothesis \ref{hyp3}.

\begin{figure}[htbp]
    \centering
    \includegraphics[scale=0.45]{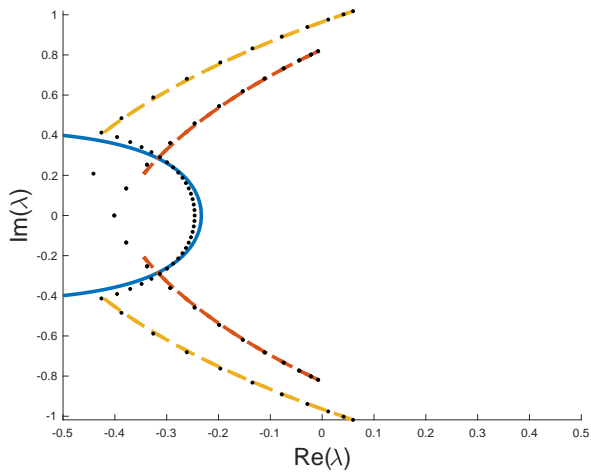}
     \caption{100 eigenvalues (dots) of the discretized and exponentially conjugated operator nearest the origin at approximately the time of bifurcation with $k=1/2$ and exponential weight $\eta=0.2$.  Overlaid are the essential spectrum of the weighted operator on an unbounded domain (solid curve), and the absolute spectrum (dashed curves).}
     \label{fig:spec_at_bif}
\end{figure}

 In numerically solving for the spectrum of the operator $L$, we are able to also find discretized approximations of eigenfunctions.  Three such eigenfunctions are shown in Figure \ref{fig:efuncs} left. The first eigenfunction corresponds to the patterns found in one dimension and comes from the most unstable branch, which reaches furthest into the right half-plane.  The other two eigenfunctions, which are transversely modulated, correspond to the rightmost eigenvalues of the next most unstable branch.  Varying $c$, we find the corresponding eigenvalues cross the imaginary axis generically for some unique $c_*$. Thus we can see numerically that there is an isolated pair of generic Hopf eigenvalues with algebraic and geometric multiplicity two, crossing at some non-zero speed $c_*$,  indicating that Hypothesis \ref{hyp5} is satisfied; see Figure \ref{fig:efuncs} center.

\begin{figure}[htbp]
\centering
\includegraphics[scale=0.275]{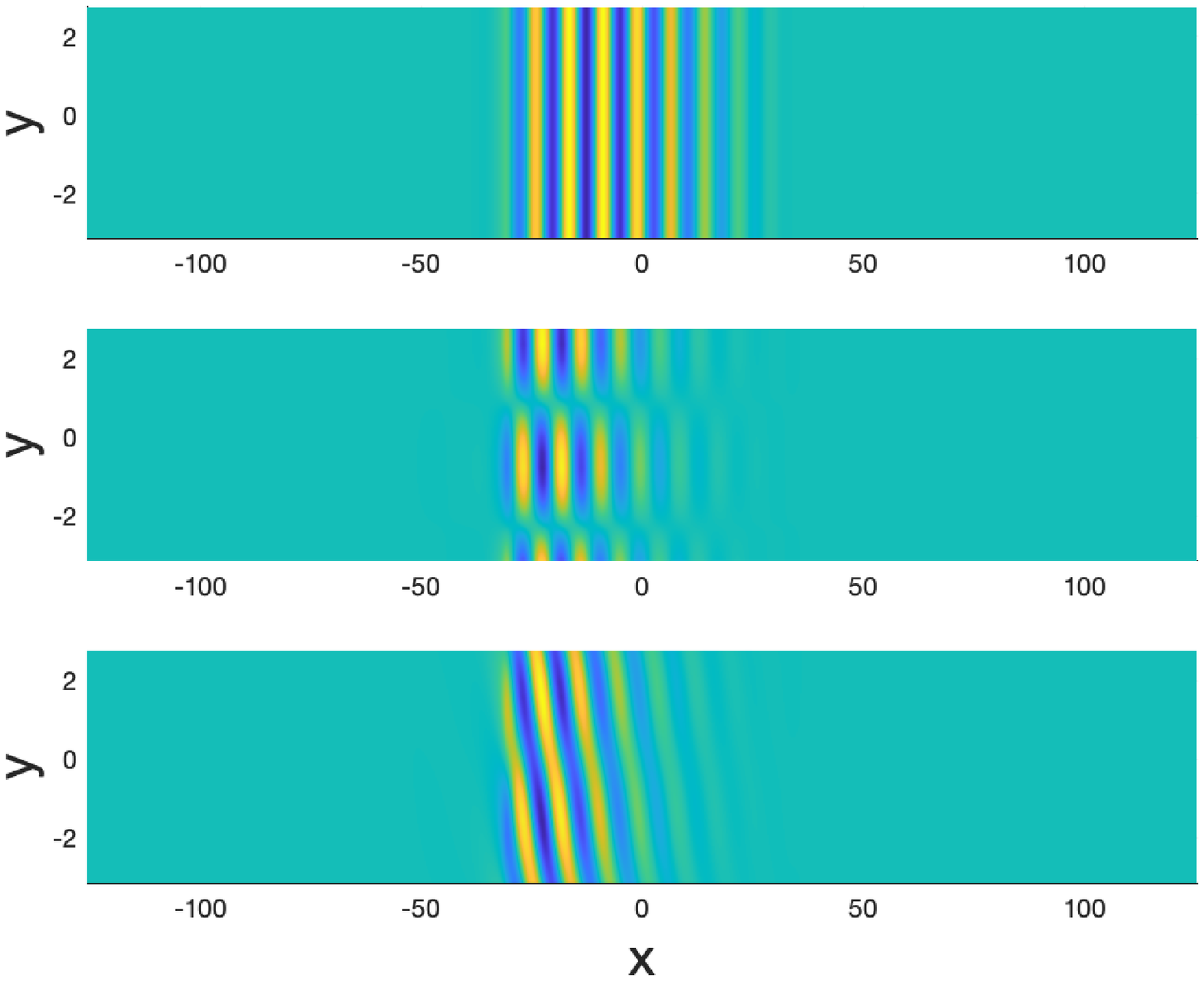}\hspace{-0.2in}
\includegraphics[scale=0.275]{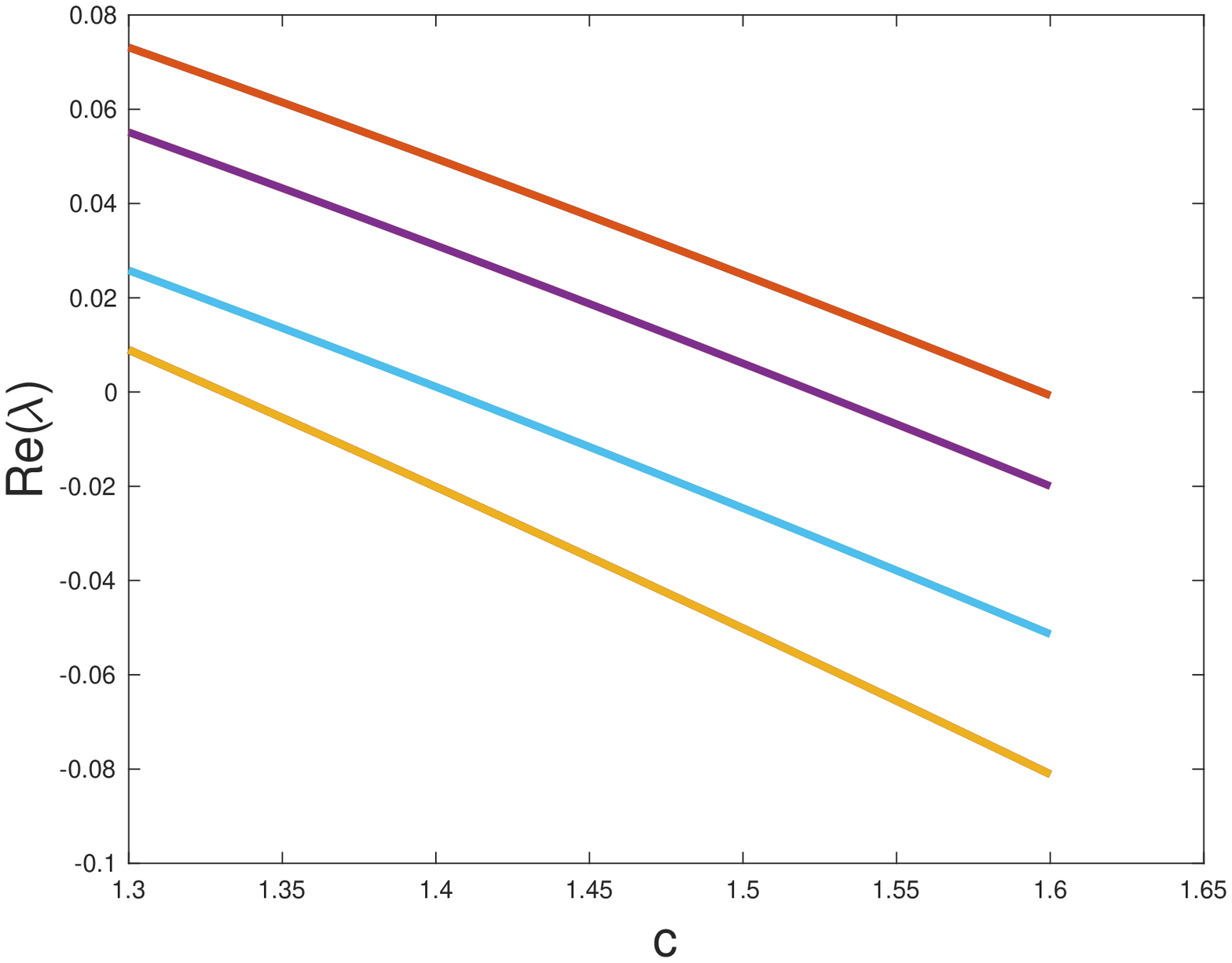}\hspace{-0.2in}
\includegraphics[scale=0.275]{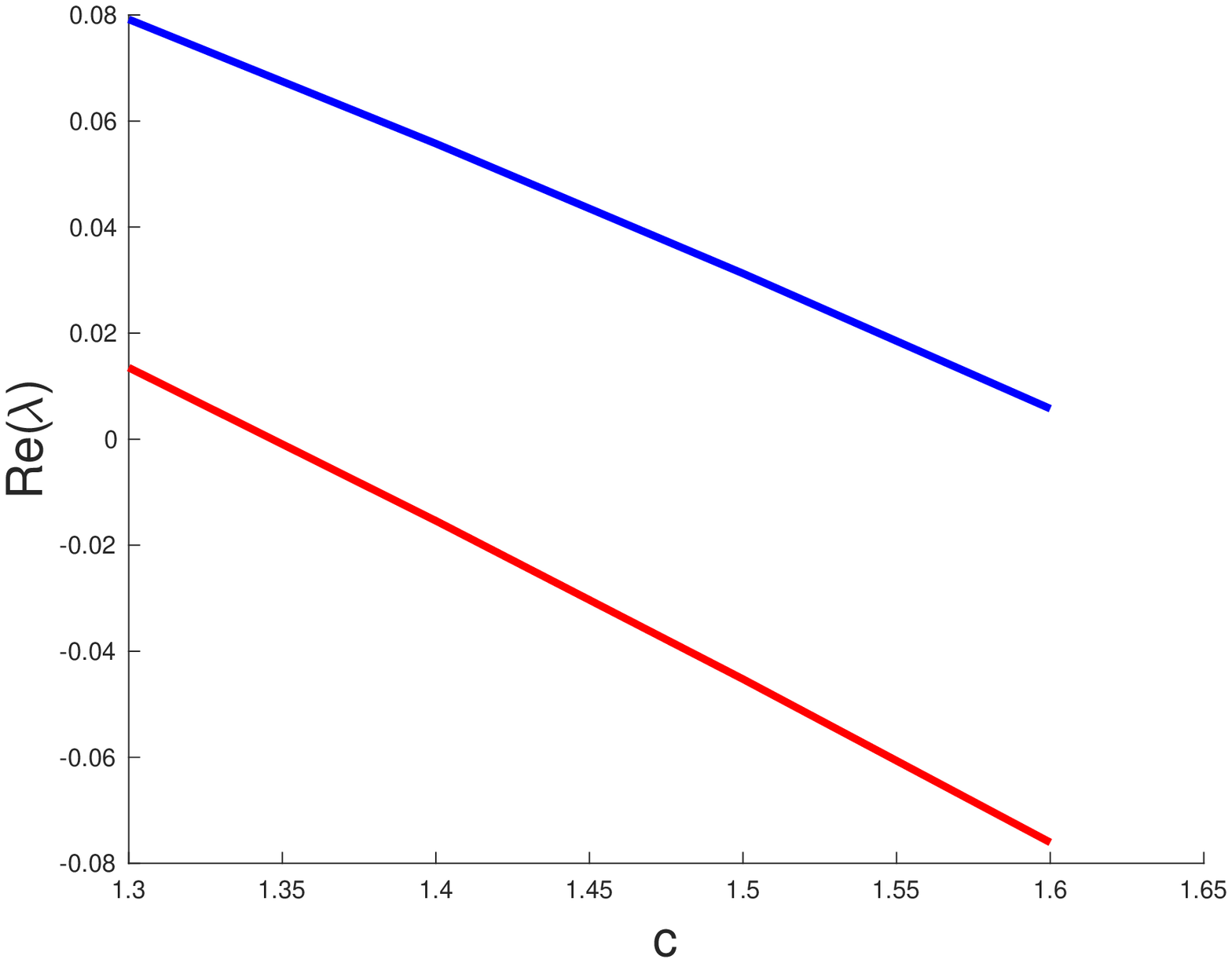}
\caption{(Left) Eigenfunctions: one depicting vertical stripes, and two others depicting transverse patterns.  (Center) The real parts of the first 10 unstable eigenvalues varying with $c$, indicating leading order Hopf bifurcation locations.  (Right) The real part of the branch points of the absolute spectrum varying with $c$.  The $\ell=1$ (lower) branch is neutral at $c_*\approx 1.35$, and the $\ell=0$ (upper) branch is neutral at $c_*\approx 1.6$.}
\label{fig:efuncs}
\end{figure}

Moving back to the unbounded domain problem, with $x\in \mathbb{R}$, we note that for $K$ large, point spectrum can be located by computing the absolute spectrum \cite{rademacher2007computing} of the plateau state, where $h \equiv 1$. Indeed, the work of \cite{sandstede2000gluing} implies that all but finitely many of the point spectrum of $L$, posed on the unbounded domain, accumulate onto the absolute spectrum of the trivial state with $h\equiv1$ with rate $O(1/K^2)$ as $K\rightarrow+\infty$. Hence branch points of the absolute spectrum give leading-order predictions for the onset of instabilities.

 To compute the absolute spectrum, we once again use the linear dispersion relation.  For each $\ell\in\mathbb{Z}$, we seek curves $(\lambda(\gamma),\nu(\gamma)),\, \gamma\in\mathbb{R}$, which solve $d(\lambda,\nu;k,\ell,c)=d(\lambda,\nu+i\gamma;k,\ell,c)=0$; see Figure \ref{fig:spec_at_bif} for computations using Mathematica. Branch points of the absolute spectrum can then be located by evaluating solutions at $\gamma = 0$.  As $c$ is decreased, we find that the branch points destabilize.  In Figure \ref{fig:spec_at_bif}, the most unstable dashed line corresponds to $\ell=0$, and the next most unstable dashed line to $\ell=\pm1$.  We can see from the figures that as the quench speed $c$ is decreased, the $\ell=0$ mode, corresponding to the $y$-independent mode, bifurcates first, followed by the $\ell=\pm1$ transverse modes.  Plotting $c$ versus $\mathrm{Re}(\lambda(0))$ we find good agreement with the numerical eigenvalues. 

Further, we can compute the speeds at which these branch points will cross the imaginary axis using the linear spreading speed calculations of \cite[Sec. 2.11]{van2003front}. Linear spreading speeds for the homogeneous system with $h\equiv 1$ can be obtained in the 2D channel by Fourier decomposing in $y$ and studying the spreading speed for each transverse modulation $e^{i\ell y},\, \ell\in\mathbb{Z}$. We obtain the following family of linearized equations
\begin{equation}
	\partial_{\tau}v=-(\partial_x^2-k^2\ell^2)[(\partial_x^2-k^2\ell^2)v+v], \quad \ell\in \mathbb{Z},
\end{equation}
 Expanding this with $k=1/2$, \cite{van2003front} gives the $\ell = 0$ and $\ell = \pm 1$ linear spreading speeds respectively as
\[
c_{*,0}=\frac{2}{3\sqrt{6}}(2+\sqrt{7})(\sqrt{7}-1)^{1/2} =  1.622...,
\qquad
c_{*,1}=\frac{7}{3\sqrt{3}}= 1.347...
\]
We find branch points for $|\ell|>1$ lie in the open left half-plane and are bounded away from the imaginary axis for speeds $c$ near the transverse Hopf speed $c_{*,1} = 1.347...$.  Since we have control of all other branch points near the transverse Hopf-speed $c_{*,1}=1.347$, we can conclude strong numerical evidence that no other resonant point spectra bifurcate at the same $c$ as the Hopf-instability.  This, along with our numerical computations and the discussion on the essential spectrum above, indicates Hypothesis \ref{hyp6} also holds.

\subsection{Bifurcations}\label{sec3.2}
Calculation of the bifurcation coefficients $\theta_1$ and $\theta_2$, and thus the direction of bifurcation, requires evaluation of the eigenfunctions, the corresponding adjoint eigenfunctions, as well as the evaluation of $u_*$ in the derivatives of $f$ (which is trivial in this case).  Instead of expanding these coefficients theoretically, we instead investigate the direction of bifurcation numerically.

Using the transverse eigenfunctions described in Figure \ref{fig:efuncs} above as initial conditions, we use direct numerical simulation, with spectral discretization in space and a Crank-Nicholson method in time, to simulate the bifurcated nonlinear states.  We then continue these solutions adiabatically, varying $c$ and letting the end-time solution of the previous $c$ value relax to a steady state for the new $c$ value before incrementing again.  This is done for both $\gamma=-1$ and $\gamma=2$, the difference of which we find as being mediation between super- and subcritical bifurcations of standing and rotating waves.  In doing so, we produce Figure \ref{fig:bif_diag}.  As these are direct numerical solutions, only locally stable states are observed.

\begin{figure}[htbp]
    \centering
    \includegraphics[scale=0.5]{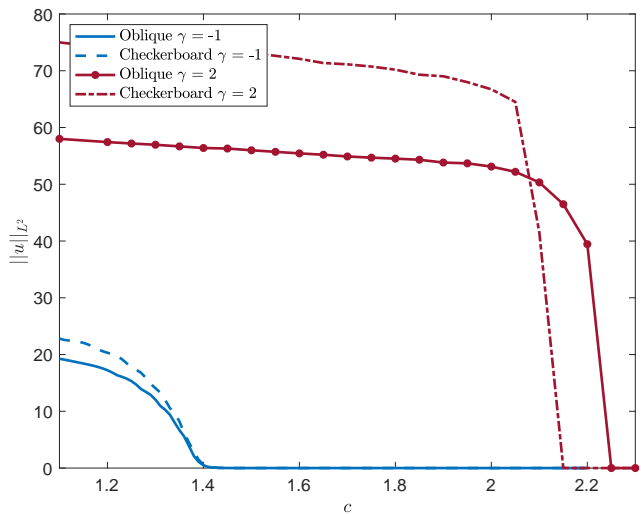}
    \caption{Numerical bifurcation diagram depicting the branching for both the pushed and pulled nonlinearities.}
    \label{fig:bif_diag}
\end{figure}

For the $\gamma=-1$, or pulled, case, we note that there is bifurcation around $c=1.4$ where the $L^2$ norm of the solution $u$ begins branching off from 0.  We would expect that the pushed case, here $\gamma=2$, would bifurcate from the same point, around $c=1.4$, with unstable subcritical branch bifurcating in $c>1.4$. While not observed, we expect this branch to continue up to some $c$ where it hits a fold point connecting with the large-amplitude nonlinear state observed in our numerics.  We also remark that the adiabatic continuations given by Figure \ref{fig:bif_diag} indicate that the folds for checkerboard and oblique stripes occur at different speeds.

\subsection{Varying $k$}\label{sec3.3}
We now wish to explore what happens as we vary the parameter $k$, which controls the vertical wavenumber of patterns.  To do this, we vary $k$ between 0 and 0.9 and find the 100 eigenvalues closest to 0 for a fixed speed, $c=1.2$.  We then find the most unstable (or least stable) transverse mode and plot its real part against $k$.  In doing so, we obtain Figure \ref{fig:k} left.

\begin{figure}[htbp]
    \centering
    \includegraphics[scale=0.4]{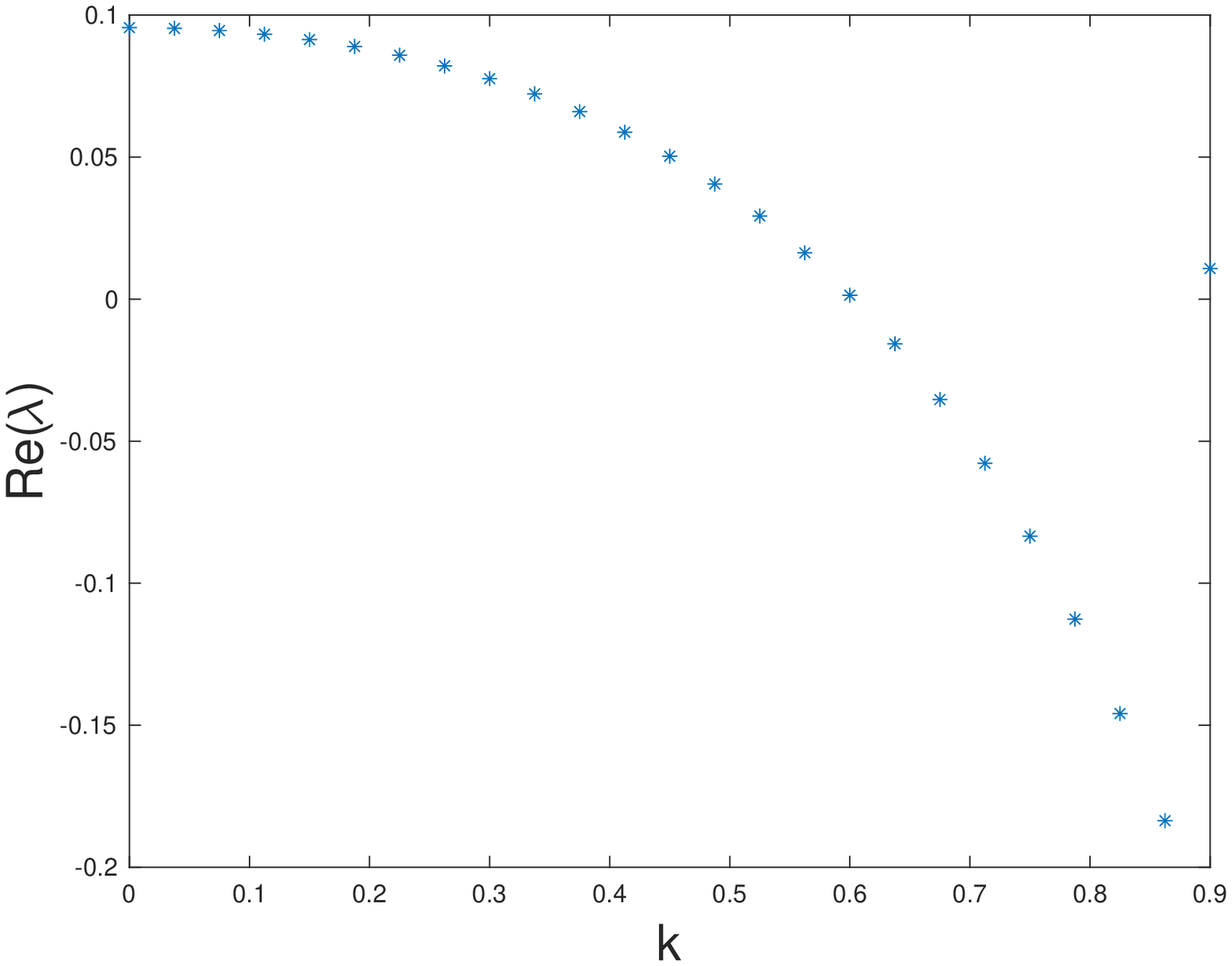}\hspace{-0.2in}
    \includegraphics[scale=0.4]{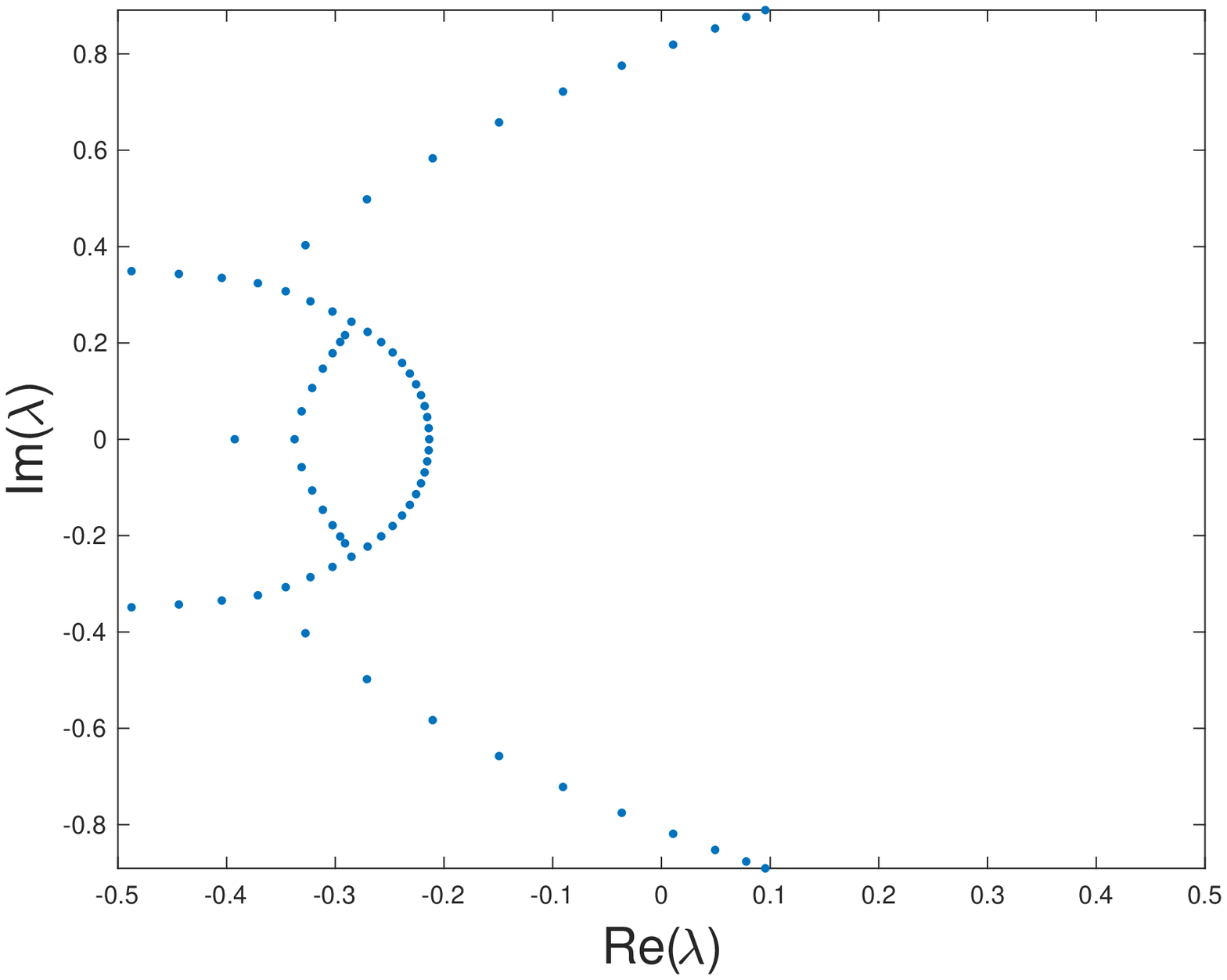}
    \caption{(Left) The real parts of the transverse eigenvalues as $k$ varies in $[0,0.9]$.  (Right) The spectrum with transverse wavenumber $k=94/99$ at speed $c=1.2$ with exponential weight $\eta=0.2$.}
    \label{fig:k}
\end{figure}

 For $c$ fixed positive, we find that as $k$ increases, and thus the vertical period decreases, the transverse mode stabilizes, indicating that no transverse patterns arise and that the transverse Hopf location happens for smaller speeds $c$. 
In Figure \ref{fig:k} right,  we depict the numerical spectrum for $k= 94/99$, observing that the transverse branches of spectrum have moved to the left of the essential spectrum, which itself is shifted due to the exponential weight.

\section{Discussion}\label{s:disc}
We have shown in Section \ref{s:ab} that, in the wake of a quench, the 2-dimensional Cahn-Hilliard equation can produce a pair of one-parameter families of time- and $y$-periodic solutions bifurcating from a given front solution under Hypotheses \ref{hyp1}-\ref{hyp6}.  Our results give leading-order forms for these solution families as well as computable formulas for the bifurcation coefficients, allowing the determination of the bifurcation direction.  In Section \ref{s:ex}, we have given an explicit example of this behavior, and shown numerically what happens as the vertical wavenumber $k$ increases.

There are several avenues of subsequent inquiry which could follow from our work. 
First of all, one naturally would wish to study how the local bifurcating branches established here continue globally in the quench speed $c$ and the vertical wavenumber $k$.  Indeed in the subcritical pushed case, $\gamma >1$, considered in Section \ref{s:ex}, one would seek to locate the secondary fold bifurcation to the large amplitude nonlinear states. We expect such a location to be mediated by the interaction of the oscillatory tail of the patterned state with the quenching interface \cite{goh2016pattern}.  Next it would be of interest to study pattern selection in a domain with large $y$ period (i.e., $k\rightarrow0^+$) where possibly several transverse modes can be excited. We expect the mode with the largest period to bifurcate first, but it would be interesting to see if subsequent bifurcations of higher harmonics lead to multi-mode interactions or defect nucleation.

In another direction, one could seek to establish transverse patterns where the spinodally unstable region is unbounded and bifurcating solutions are asymptotically periodic as $x\rightarrow-\infty$, for example taking a step-function like quench $h(\tilde x) = -\tanh(\delta\tilde x)$ in \eqref{e:fex}. One possible approach would be to first consider a heterogeneity of the form $h_K(\tilde x) = \tanh(\delta \tilde x) \tanh(-\delta(\tilde x + K))$, covered by our hypotheses, and take the large plateau limit $K\rightarrow+\infty$ to establish a full pattern forming front.

Next, there are several unanswered questions on stability of such fronts. Indeed, the stability of the parallel striped fronts, posed in either 1- or 2-dimensional spatial domains, has not been established. We expect a reduced stability principle \cite{kielhofer2011bifurcation} to provide a relatively straightforward approach to establishing stability in one dimension. Moving to the transversely modulated patterns studied in this work, one cannot use such reduced stability principles as the trivial state from which they bifurcate is already unstable due to the parallel striped Hopf instability. Hence we expect the transverse patterns to be unstable or metastable near the bifurcation point. Indeed, since we do observe these patterns numerically, it would be of interest to understand how initial conditions starting near the transverse patterns dynamically evolve, and how or whether they converge to another state such as the parallel striped fronts for long times. It would also be of interest to periodically extend both the parallel and transverse patterns in $y\in \mathbb{R}$ and consider stability to localized $L^2(\mathbb{R}^2)$ perturbations.

\appendix
\section{Fredholm Properties}\label{app1}

In this appendix, we provide the proof of Proposition \ref{pfred}.  The general approach will be to apply an abstract closed range lemma to the linear operator and its $\mathcal{X}$-adjoint to obtain that the linearization $\mathcal{L}$ is Fredholm. We then compute its index via a Fourier decomposition in $\tau$ and $y$. To begin, for $J>0$, let $\mathcal{X}(J)$ and $\mathcal{Y}(J)$ denote the spaces of functions, in $\mathcal{X}$ and $\mathcal{Y}$ respectively, which have $x$-support in the interval $[-J,J]$.  Since the embedding $\mathcal{Y}(J)\hookrightarrow\mathcal{X}(J)$ is compact, we have the following:

\begin{Lemma}\label{l:a1}
    There exist constants $C>0$ and $J>0$ such that the operator $\mathcal{L}$ as defined in Section 2 satisfies
    \begin{equation}
        \|\xi\|_\mathcal{Y}\leq C(\|\xi\|_{\mathcal{X}(J)}+\|\mathcal{L}\xi\|_\mathcal{X}),\quad \xi\in \mathcal{Y}.
    \end{equation}
\end{Lemma}

\begin{proof}
We remark that the proof of this result follows a similar approach as \cite[Lem. 2.3]{Goh} but we include it for completeness. Throughout $C>0$ will be a changing constant, possibly dependent on the weight $\eta$, the front $u_*$, parameters $c_*,\omega_*$ and the nonlinearity $f$, but not $\xi$. 

    \textbf{Step 1:} We begin by proving that the estimate holds for $J=\infty$.  Assume for the moment that we have the exponential weight $\eta=0$. Then
  \begin{equation}
        \|\mathcal{L}\xi\|_{\mathcal{X}}\geq\|(\partial_\tau+\Delta_k^2)\xi\|_{\mathcal{X}}-\|\Delta_k(\partial_uf(x,u_*)\xi)-c_*\partial_x\xi\|_{\mathcal{X}}.
    \end{equation}
    Since $f$ and $u_*$ are smooth, we have for all $\epsilon>0$\begin{align}
        \|\Delta_k(\partial_uf(x,u_*)\xi)-c_*\partial_x\xi\|_{\mathcal{X}}&\leq C\|\xi\|_{H^2(\mathbb{R}\times\mathbb{T},X)}\nonumber\\
        &\leq C\|\xi\|^{1/2}_{\mathcal{X}}\cdot\|\xi\|^{1/2}_{H^4(\mathbb{R}\times\mathbb{T},X)}\nonumber\\
        &\leq C(\epsilon\|\xi\|_{H^4(\mathbb{R}\times\mathbb{T},X)}+\frac{1}{4\epsilon}\|\xi\|_{\mathcal{X}}).
    \end{align}
    By combining the two inequalities, we see that for $\epsilon$ sufficiently small\begin{align}
        \|\mathcal{L}\xi\|_{\mathcal{X}}+\frac{C}{4\epsilon}\|\xi\|_{\mathcal{X}}&\geq\|(\partial_\tau+\Delta_k^2)\xi\|_{\mathcal{X}}-C\epsilon\|\xi\|_{H^4(\mathbb{R}\times\mathbb{T},X)}\nonumber\\
        &\geq C_1\|\xi\|_{\mathcal{Y}}.
    \end{align}
    For $\eta>0$, one follows a similar procedure with the conjugated operator $\mathcal{L}_\eta:=e^{\eta\langle x\rangle}\mathcal{L}e^{-\eta\langle x\rangle}$.  Additional terms which arise from the conjugation are small because the weight $\eta$ is small.
    
    \textbf{Step 2:} Next, we wish to show that the estimate holds for the constant coefficient operators \begin{equation}
        \mathcal{L}_{\pm}\xi=\omega_*\partial_\tau \xi+\Delta_k(\Delta_k \xi+f_\pm'(u_\pm)\xi)-c_*\partial_x \xi.
    \end{equation}
    Again, we must work with the conjugated operators $\mathcal{L}_{\pm,\eta}:=e^{\eta\langle x\rangle}\mathcal{L}_\pm e^{-\eta\langle x\rangle}$.  If $\mathcal{L}_{\pm,\eta}\xi=h$, then by taking the Fourier transform in $x,y,$ and $\tau$ we see that
    \begin{align}
        &\hat{h}(i\zeta,i\chi,i\rho)=\nonumber\\
       & \left[i\rho\omega_*-i(\zeta-\eta)c_*+\left((\zeta-\eta)^2+k^2\chi^2\right)^2-\left((\zeta-\eta)^2+k^2\chi^2\right)f'_\pm(u_\pm)\right]\hat{\xi}(i\zeta,i\chi,i\rho),\\
         &\zeta\in\mathbb{R}, \chi,\rho\in\mathbb{Z}.\nonumber 
    \end{align}
    By Hypothesis \ref{hyp6}, for $\eta>0$ the essential spectrum of the time-independent portion of the operator will not intersect $i\omega_*\mathbb{Z}$.  Thus both equations $\mathcal{L}_{\pm,\eta}\xi=h$ are invertible, and thus so are their Fourier transforms.  Using this, we get\begin{equation}
        \hat{\xi}=[i\rho\omega_*-i(\zeta-\eta)c_*+((\zeta-\eta)^2+k^2\chi^2)^2-((\zeta-\eta)^2+k^2\chi^2)f'_\pm(u_\pm)]^{-1}\hat{h}.
    \end{equation}
    The coefficient on the right-hand side must be bounded by our assumptions, and so we have
    \begin{align}
        \|\hat{\xi}\|_{\mathcal{Y}}&\leq\sup_{\zeta,\chi,\rho}\left|[i\rho\omega_*-i(\zeta-\eta)c_*+((\zeta-\eta)^2+k^2\chi^2)^2-((\zeta-\eta)^2+k^2\chi^2)f'_\pm(u_\pm)]^{-1}\right|\|\hat{h}\|_{\mathcal{X}}\nonumber\\
        &=C\|\widehat{\mathcal{L}_{\pm,\eta}\xi}\|_\mathcal{X}.
    \end{align}
    By Plancherel's Theorem, this gives us $\|\xi\|_{\mathcal{Y}}\leq C\|\mathcal{L}_\pm\xi\|_{\mathcal{X}}$.
    
    \textbf{Step 3:} Finally, we seek to complete the proof by using the estimates previously established to perform a patching argument.  For $J>1$, let $\xi^{\pm}\in\mathcal{Y}$ be such that $\xi^+(x,y)=0$ for all $x\leq J-1$ and $\xi^-(x,y)=0$ for all $x\geq 1-J$.  The exponential convergence rates from Hypotheses \ref{hyp1} and \ref{hyp2} as $x\to\pm\infty$ give the following: for every $\epsilon>0$ there is some $J>0$ sufficiently large such that\begin{equation}
        \|(\mathcal{L}_\pm-\mathcal{L})\xi^\pm\|_\mathcal{X}\leq\epsilon\|\xi^\pm\|_{H^2_\eta(\mathbb{R}\times\mathbb{T}_y,X)}.
    \end{equation}
    From this and the estimate from Step 2, we get that
\begin{align}
        \|\xi^\pm\|_\mathcal{Y}&\leq C\|\mathcal{L}_\pm\xi^\pm\|_\mathcal{X}\nonumber\\
        &\leq C(\|(\mathcal{L}_\pm-\mathcal{L})\xi^\pm\|_\mathcal{X}+\|\mathcal{L}\xi^\pm\|_\mathcal{X})\nonumber\\
        &\leq C(\epsilon\|\xi^\pm\|_{H^2_\eta(\mathbb{R}\times\mathbb{T}_y,X)}+\|\mathcal{L}\xi^\pm\|_{\mathcal{X}}).
\end{align}
    Choosing $\epsilon<\frac{1}{C}$, we see that
    \begin{equation}\label{e:a11}
        \|\xi^\pm\|_\mathcal{Y}\leq C\|\mathcal{L}\xi^\pm\|_\mathcal{X}.
    \end{equation}
    Next we consider an element $\xi\in\mathcal{Y}$ with $\xi=0$ for all $|x|\leq J-1$. Then, we can decompose $\xi=\xi^++\xi^-$, where
    \begin{equation}
        \xi^+=\left\{\begin{array}{cc}
             \xi(x),&x\geq 0  \\
             0,&x<0 
        \end{array}\right.,\xi^-=\left\{\begin{array}{cc}
             0,& x\geq0 \\
             \xi(x),&x<0 
        \end{array}\right..\nonumber
    \end{equation}
    Applying estimate \eqref{e:a11} and the triangle inequality, we obtain
    \begin{align}
        \|\xi\|_\mathcal{Y}^2&\leq\|\xi^+\|^2_\mathcal{Y}+\|\xi^-\|^2_\mathcal{Y}\nonumber\\
        &\leq C(\|\mathcal{L}\xi^+\|^2_\mathcal{X}+\|\mathcal{L}\xi^-\|^2_\mathcal{X})\nonumber\\
        &\leq C(\|\mathcal{L}\xi\|^2_\mathcal{X}+\|\mathcal{L}\xi\|^2_\mathcal{X})=C\|\mathcal{L}\xi\|^2_\mathcal{X}.\label{e:a12}
    \end{align}
    Lastly, for a general $\xi\in\mathcal{Y}$, we choose a smooth bump function $\beta$ such that $\beta=1$ when $|x|\leq J-1$ and $\beta=0$ for $|x|\geq J$.  From the triangle inequality (first line), the results of Step 1 (second line), and estimate \eqref{e:a12} (second line), we see that\begin{align}
        \|\xi\|_\mathcal{Y}&\leq\|\beta\xi\|_\mathcal{Y}+\|(1-\beta)\xi\|_{\mathcal{Y}}\nonumber\\
        &\leq C(\|\beta\xi\|_\mathcal{X}+\|\mathcal{L}(\beta\xi)\|_\mathcal{X}+\|\mathcal{L}((1-\beta)\xi)\|_\mathcal{X})\nonumber\\
        &\leq C(\|\xi\|_{\mathcal{X}(J)}+\|\mathcal{L}\xi\|_\mathcal{X}).\nonumber
    \end{align}
\end{proof}

We then have the following corollary:
\begin{Corollary}
    $\mathcal{L}$ has closed range and finite dimensional kernel.
\end{Corollary}

\begin{proof}
    Since the embedding $\mathcal{Y}(J)\hookrightarrow\mathcal{X}(J)$ is compact, we have that the identity operator is compact.  Then, the proof follows by applying an abstract closed range lemma, such as in \cite[Ch. 6, Prop 6.7]{Taylor}.
\end{proof}

We can define the $L^2$-adjoint by integration by parts, finding $\mathcal{L}^*=-\omega_*\partial_\tau+\Delta_k^2+\partial_uf(x,u_*)\Delta_k+c_*\partial_x$.  Since we wish to work with exponentially weighted spaces, we must define the $L^2_\eta$-adjoint.  This is done by once again working with conjugated operators posed on $L^2$.  Recall the conjugated operator $\mathcal{L}_\eta$, posed on $L^2$, is given by $\mathcal{L}_\eta:=e^{\eta\langle x\rangle}\mathcal{L}e^{-\eta\langle x\rangle}$.
Because of this, we have 
\begin{align}
    \langle \mathcal{L}_\eta u,v\rangle_{L^2}&=\langle e^{\eta\langle x\rangle}\mathcal{L}e^{-\eta\langle x\rangle}u,v\rangle_{L^2}=\nonumber\\
    &=\frac{1}{4\pi^2}\int_0^{2\pi}\int_0^{2\pi}\int_{-\infty}^\infty e^{\eta\langle x\rangle}\mathcal{L}(e^{-\eta\langle x\rangle}u)\overline{v}dxdy d\tau=\nonumber\\
    &=\frac{1}{4\pi^2}\int_0^{2\pi}\int_0^{2\pi}\int_{-\infty}^\infty e^{-\eta\langle x\rangle}u\mathcal{L}^*(\bar v e^{\eta\langle x\rangle})dxdyd\tau=\nonumber\\
    &=\langle u,e^{-\eta\langle x\rangle}\mathcal{L}^*e^{\eta\langle x\rangle}v\rangle_\mathcal{X}=\langle u,\mathcal{L}^*_\eta v\rangle_{\mathcal{X}}
\end{align}
and hence \begin{equation}\mathcal{L}^*_\eta:=e^{-\eta\langle x\rangle}\mathcal{L}^*e^{\eta\langle x\rangle}.\nonumber\end{equation}

Hence $\mathcal{L}^*$ is defined on a weighted $L^2$ space with weight $e^{-\eta\langle x\rangle}$.  This conjugated operator can be run through the same estimates as in Lemma \ref{l:a1} for $\eta$ sufficiently small, and we reach the conclusion of the corollary.  Thus $\mathcal{L}:\mathcal{Y}\to\mathcal{X}$ is a Fredholm operator.

To find the Fredholm index, we first form the Fourier series in $\tau$ and $y$ to get
\begin{equation}
    u(x,y,\tau)=\sum_{\ell_\tau,\ell_y}e^{i\ell_\tau\tau}e^{i\ell_yy}\hat{u}_{\ell_\tau,\ell_y}(x),
\end{equation}
where $\hat{u}_{\ell_\tau,\ell_y}\in L^2_\eta(\mathbb{R})$, 
 as well as the decomposition of $\mathcal{X}$ as $\bigoplus_{\ell_\tau,\ell_y}\mathcal{X}_{\ell_\tau,\ell_y}$, where $\mathcal{X}_{\ell_\tau,\ell_y}:=\{e^{i\ell_\tau\tau}e^{i\ell_yy}\hat{u}_{\ell_\tau,\ell_y}(x),\,\,|\,\hat{u}_{\ell_\tau,\ell_y}\in L^2_\eta(\mathbb{R})\}$; see \cite{AnB} for more detail.  This induces a decomposition of $\mathcal{Y}$ as $\mathcal{Y}=\bigoplus_{\ell_\tau,\ell_y}\mathcal{Y}\cap\mathcal{X}_{\ell_\tau,\ell_y}=\bigoplus_{\ell_\tau,\ell_y}\mathcal{Y}_{\ell_\tau,\ell_y}$.  We then define
\begin{align}
\mathcal{L}_{\ell_\tau,\ell_y}:\,\, \mathcal{Y}_{\ell_\tau,\ell_y}\subset\mathcal{X}_{\ell_\tau,\ell_y}&\to\mathcal{X}_{\ell_\tau,\ell_y}\nonumber\\
   \hat u&\mapsto (\partial_x^2-k^2\ell_y^2)[(\partial_x^2-k^2\ell_y^2)\hat u+\partial_uf(x,u_*)\hat u]-\left(c_*\partial_x+i\omega_*\ell_\tau\right) \hat u.\nonumber
\end{align}

We first organize this decomposition into three subspaces, 
$
\mathcal{X}_{0,0}, 
 \bigoplus_{|\ell_\tau|=1, |\ell_y| = 1} \mathcal{X}_{\ell_\tau,\ell_y}, 
 $ 
 and their complement, $\mathcal{X}_h = X_{1,0}\oplus X_{0,1}\oplus X_{2,1}\oplus X_{1,2}\oplus \left( \bigoplus_{|\ell_y|,|\ell_\tau|\geq2} X_{\ell_y,\ell_\tau}\right)$, in $\mathcal{X}$.
\
\begin{Lemma}\label{l:00}
    $\mathrm{ind}\,\mathcal{L}_{0,0}=-1$.
\end{Lemma} 
\begin{proof}
  Recall, $\mathcal{L}_{0,0}=\partial_x^2[\partial_x^2+\partial_uf(x,u_*)]-c_*\partial_x=-L = -\partial_x\circ\Tilde{L}$,
where $\Tilde{L}:=-\partial_x(\partial_x^2+\partial_uf(x,u_*))+c_*$ converges to the constant coefficient operators $\Tilde{L}_{\pm}$ as $x\rightarrow\pm\infty$ by our hypotheses. For $c>0$, each of the polynomials $\nu^3+f'_\pm(u_\pm)\nu-c=0$ has two positive roots and one negative root. Thus the difference between the number of unstable eigenvalues is zero and so $\Tilde{L}$ has Fredholm index 0.  Then since $\partial_x$ has Fredholm index -1, we have that $\mathrm{ind}\,\mathcal{L}_{0,0}=-1$.
\end{proof}

\begin{Lemma}\label{L:11}
    For $|\ell_\tau|=|\ell_y|=1$, $\mathrm{ind}\,\mathcal{L}_{\ell_\tau,\ell_y}=0$.
\end{Lemma}
\begin{proof}
There are four index pairs considered here, depending on the signs of $\ell_\tau$ and $\ell_y$.  We only show the case $\ell_\tau=\ell_y=1$, as the other three cases follow the same reasoning. Here we have 
$$
\mathcal{L}_{1,1}=(\partial_x^2-k^2)[(\partial_x^2-k^2)+\partial_uf(x,u_*)]-c_*\partial_x+i\omega_*\partial_\tau = -L_{1,1} + i\omega_*\partial_\tau
$$
with the spatial operator $L_{1,1}:=-(\partial_x^2-k^2)[(\partial_x^2-k^2)+\partial_uf(x,u_*)]+c_*\partial_x$. Hypothesis \ref{hyp5} implies that $L_{1,1}$ has a eigenvalue $\lambda = i \omega$ with one-dimensional eigenspace spanned by $e^{i(y+\tau)}p(x)$, and similarly for the adjoint $L_{1,1}^*$ for the eigenvalue $-i\omega_*$;  see \cite[section 1.5.5]{Kato}.  Thus the kernel of both $\mathcal{L}_{1,1}$ and its adjoint operator $\mathcal{L}_{1,1}^*$ is one-dimensional and hence $\mathrm{ind}\,\mathcal{L}_{1,1}=0$.
\end{proof}

\begin{Lemma}
Defining $\mathcal{L}_h := \mathcal{L}\big|_{\mathcal{X}_h}$, we have $\mathrm{ind}\, \mathcal{L}_h = 0$. 
\end{Lemma}
\begin{proof}
First, we recall that for $\lambda\in i\mathbb{Z}\diagdown \{\pm i\omega_*,0\}$, $L - \lambda$ and thus $L^* - \lambda$ is invertible. This implies that for any $\ell_\tau\in \mathbb{Z}\diagdown \{0,\pm1\}$ and $\ell_y\in \mathbb{Z}$, that $\mathrm{ind}\, \mathcal{L}_{\ell_t,\ell_y} = 0$. 

Next, operators $\mathcal{L}_{\ell_t,\ell_y}$ with $\ell_\tau = 0$ also have index 0 due to Hypothesis \ref{hyp3} which gives that $\lambda = 0$ is not in the extended point spectrum.   Finally, operators with $\ell_\tau = \pm1$ but $|\ell_y|\neq1$ have index 0 due to Hypothesis \ref{hyp5} which gives that the Hopf eigenfunctions lie in the $\ell_y = \pm 1$ subspaces. 
\end{proof}

Combining the above Lemmas, we use standard Fredholm algebra results \cite{Taylor} to conclude 
\begin{Proposition}
    $\mathrm{ind}\,\mathcal{L}=-1$ for $\mathcal{L}:\mathcal{Y}\to\mathcal{X}$.
\end{Proposition}

Now we consider a closed subset defined $\mathring{\mathcal{X}}:=\{u\in\mathcal{X}|\langle u,e^{-2\eta\langle x\rangle}\rangle_\mathcal{X}=0\}$.    We note that for any exponential weight $\eta>0$ and for any $u\in\mathcal{X}$
\begin{align}
    \langle u,e^{-2\eta\langle x\rangle}\rangle_\mathcal{X}&=\frac{1}{4\pi^2}\int_0^{2\pi}\int_0^{2\pi}\int_{-\infty}^\infty u(x,y,\tau)e^{-2\eta\langle x\rangle}e^{2\eta\langle x\rangle}dxdy d\tau=\frac{1}{4\pi^2}\int_0^{2\pi}\int_0^{2\pi}\int_{-\infty}^\infty u(x,y,\tau)dxdy d\tau\nonumber
\end{align}

For any $v=\mathcal{L}u, u\in\mathcal{Y}$ we have $\langle v,e^{-2\eta\langle x\rangle}\rangle_\mathcal{X}$=0 by integration by parts, and so $v\in\mathring{\mathcal{X}}$.  Thus we have that $\mathcal{L}$ maps $\mathcal{Y}$ into $\mathring{\mathcal{X}}$.  Finally, by composing $\mathcal{L}:\mathcal{Y}\rightarrow \mathcal{X}$ with the index 1 orthogonal projection $\mathcal{P}:\mathcal{X}\rightarrow\mathring{\mathcal{X}}$, Fredholm algebra gives that $\mathcal{L}:\mathcal{Y}\rightarrow \mathring{\mathcal{X}}$ has Fredholm index 0.

This concludes the proof of Proposition \ref{pfred}.

\section{Hopf Bifurcation with $O(2)$ symmetry}\label{app2}

The contents of this section can be found in detail in \cite[Ch. 16]{Golubitsky}. As was stated in the hypotheses, the presence of symmetry forces generic Hopf eigenvalues to have algebraic and geometric multiplicity 2.  Hence an equivariant Hopf theorem is needed. In this appendix, we lay out an introduction to this theorem.

Suppose we have a bifurcation problem with $\Gamma$-symmetry, i.e., we have \[\partial_{\tau}u+G(u;c)=0\] where $G:\mathbb{R}^n\times\mathbb{R}\to\mathbb{R}^n$, and $G(\gamma\cdot u;c)=\gamma\cdot G(u;c)$ for $\gamma\in\Gamma$, with $c$ being our bifurcation parameter.  We say that the space $\mathbb{R}^n$ is \textbf{$\Gamma$-simple} if either of the following conditions holds: (1) we can write $\mathbb{R}^n=V\oplus V$ for some subspace $V$ where linear mappings $F:V\to V$ such that $$F(\gamma\cdot v)=\gamma\cdot F(v), \forall\,\, v\in V, \gamma\in\Gamma$$ are multiples of the identity (called absolutely irreducible), or (2) $\mathbb{R}^n$ is such that the only $\Gamma$-invariant subspaces are $\{0\}$ and $\mathbb{R}^n$, but it does not meet condition (1) (called non-absolutely irreducible).

\begin{Lemma}\label{lem3}
    Suppose that $\mathbb{R}^n$ is $\Gamma$-simple, that $G$ commutes with the action of $\Gamma$, and that $(dG)_{0;0}$ has $i$ as an eigenvalue.  Then
    
    (a) The eigenvalues of $(dG)_{0;c}$ consist of a complex conjugate pair $\mu(c)\pm i\kappa(c)$, each of multiplicity $m=n/2$.  Moreover, $\mu$ and $\kappa$ are smooth functions of $c$.
    
    (b) There is an invertible linear map $S:\mathbb{R}^n\to\mathbb{R}^n$, commuting with $\Gamma$, such that \[(dG)_{0;0}=SJS^{-1}\qquad \text{where}\qquad J=\begin{pmatrix}
        0&-I_m\\I_m&0
    \end{pmatrix}.\]
\end{Lemma}

For a proof, see \cite[Ch. 16]{Golubitsky}.  Lemma \ref{lem3}(a) gives that eigenvalues of multiplicity $m$ cross with nonzero speed. For $m>1$, an extension of the standard Hopf theorem is needed.

The amount of symmetry present in a solution $u$ to the system is measured by the isotropy subgroup
\[\Sigma_u=\{\sigma\in \Gamma\,|\,\sigma\cdot u=u\}.\]  We also have the space of solutions fixed by the isotropy subgroups, \[\mathrm{Fix}(\Sigma_u)=\{v\in\mathbb{R}^n\,|\,\sigma\cdot v=v\,\,\forall\sigma\in\Sigma_u\}.\]  Using these, we then have the following theorem:

\begin{Theorem}
    Let $\Sigma_u$ be an isotropy subgroup of a group $\Gamma$ such that $\dim \mathrm{Fix}(\Sigma_u)=2$.  Assume that \[(dG)_{0;0}=J\] meaning that $\Gamma$ acts absolutely irreducibly.  Further assume that the eigenvalue crossing condition $\mu'(0)\neq0$ holds.  Then there is a unique branch of small-amplitude periodic solutions to the bifurcation problem of period near $2\pi$ whose spatial symmetries are given by $\Sigma_u$.
\end{Theorem}

For a proof, see \cite[Ch. 16]{Golubitsky}.  Because this theorem has found solutions which are periodic in time, the symmetry group of the solutions is not just $\Sigma_u$, but $\Sigma_u\times S^1$, taking into account the time translation symmetry.  This tells us that the full group of symmetries of the problem will be $\Gamma\times S^1$, which now not only accounts for the spatial symmetries, but also the temporal periodicity.

In our problem of the Cahn-Hilliard equation in two dimensions we note that from the reduced equations achieved by the Lyapunov-Schmidt reduction, our bifurcation equation becomes real four dimensional (considering $\Tilde{\omega}$ to be constant and taking $\Tilde{c}$ as our bifurcation parameter), mapping $\mathbb{R}^4\times\mathbb{R}\to\mathbb{R}^4$.  Thus we have $n=4$, and hence $m=2$.  Additionally, there are two different symmetries in the transverse spatial variable: the translation symmetry which corresponds to rotations, and the $y$-reflection symmetry, which corresponds to the standing waves.  Thus we have the symmetry group $\Gamma=O(2)$ in our case, and we must consider the isotropy subgroups of $O(2)\times S^1$.

Notably, there are two maximal isotropy subgroups of $O(2)\times S^1$: one corresponding to rotations denoted $\widetilde{SO}(2)=\{(\theta,\theta)\,\,|\,\,\theta\in S^1\}$, and one corresponding to reflections denoted $\mathbb{Z}_2\oplus\mathbb{Z}_2^c$.  The rotation corresponds to rotating waves, which appear as oblique stripes.  The reflection corresponds to standing waves, which appear as checkerboard patterns.

\bibliographystyle{plain}
\bibliography{Oblique_and_checkerboard_patterns_in_the_quenched_Cahn-Hilliard_model}

\end{document}